\documentclass[12pt,reqno]{amsart}
\usepackage[mathscr]{euscript}
\usepackage{amssymb, amsthm, amsmath}
\usepackage[nobysame, abbrev]{amsrefs} 
\usepackage[margin=1in]{geometry}
\usepackage{comment}
\usepackage{xcolor}
\usepackage{longtable} %for the table to span several pages
\usepackage{tikz-cd}
\usepackage{calc}
\usepackage{enumerate}
\usepackage{caption}

\usepackage{hyperref}
\hypersetup{
	colorlinks=true,
	citecolor=red,    
	linkcolor=blue  
}

\newcommand{\cal}{\mathcal}
\newcommand{\ovl}{\overline}
\newcommand{\tnm}{\textnormal}
\newcommand{\mb}{\mathbf}

\newcommand{\C}{\mathbb C}
\newcommand{\F}{\mathbb F}
\newcommand{\R}{\mathbb R}

\newcommand{\N}{\mathbb N}

\newcommand{\la}{\langle}
\newcommand{\ra}{\rangle}

\newcommand{\al}{\alpha}
\newcommand{\be}{\beta}
\newcommand{\ga}{\gamma}
\newcommand{\Ga}{\Gamma}
\newcommand{\del}{\delta}
\newcommand{\Del}{\Delta}
\newcommand{\lam}{\lambda}

\newcommand{\sig}{\sigma}

\newcommand{\vp}{\varphi}
\newcommand{{\ve}}{\varepsilon}

\newcommand{\pt}{\partial}

\newcommand{\U}{{\bf U}}

\newcommand{\SO}{{\bf SO}}

\newtheorem{thm}{\textsc{Theorem}}[section]
\newtheorem{lem}[thm]{Lemma}
\newtheorem{prop}[thm]{Proposition}

\theoremstyle{definition}
\newtheorem{defi}[thm]{Definition}
\theoremstyle{remark}

\newtheorem{rem}[thm]{Remark}

\numberwithin{equation}{section}
\numberwithin{figure}{section}
\title{Large-$N$ Limit of the Segal--Bargmann Transforms on the Spheres}
\author{Luan M. Doan}
\email{ldoan3@nd.edu}
\address{University of Notre Dame du Lac, Department of Mathematics, IN 46556, USA}
\date{\today}
\thanks{The author would like to thank Professor Brian Hall and Dr. Evan O'Dorney for helpful comments and suggestions.}

\begin{document}

%===============
\begin{abstract}
	We study the large-$N$ limit of the Segal--Bargmann transform on $S^{N-1}(\sqrt N)$, the $(N-1)$-dimensional sphere of radius $\sqrt N$, as a unitary map from the space of square-integrable functions with respect to the normalized spherical measure onto the space of holomorphic square-integrable functions with respect to a certain measure on the quadric. In particular, we give an explicit formulation and describe the geometric models for the limit of the domain, the limit of the range, and the limit of the transform when $N$ tends to infinity. We show that the limiting transform is still a unitary map from the limiting domain onto the limiting range.        
\end{abstract}
%==============

\maketitle

%=====================
\section{Introduction}\label{chap:intro}
%=====================
	The Segal--Bargmann transforms, also known as the coherent state transforms, have been popular subjects of study in mathematical physics since the second half of the last century. The original formulation of the transforms was motivated by quantum mechanics and was developed by Bargmann \cites{Bar1,Bar2} and Segal \cites{Se1,Se2} for the Euclidean case $\R^N$. From then on, there have been several ways of generalization of the transforms to different manifolds. In this work, we use the formulation in Hall \cite{Hall94} and Stenzel \cite{Stenzel99} for Euclidean spaces, compact Lie groups, and compact symmetric spaces, which can be described briefly as follows: for any function $f\colon X\to \C$ that is square-integrable with respect to a certain measure on a Lie group or a symmetric space $X$, the Segal--Bargmann transform of $f$ is a holomorphic function on the complexified space $X_\C$ of $X$, obtained by first applying the heat operator to $f$, then analytically continuing the result to the whole $X_\C$. A special case where this construction works is the sphere $S^{N-1}$, with explicit formulas worked out by Hall and Mitchell \cite{HallMit2002}. 
	
	One of the recent interesting topics of research is to study the large-$N$ limits of different families of Segal--Bargmann transforms, such as the works by Biane  \cite{Biane97} and by Driver, Hall and Kemp \cite{DHK2013} for $\U(N)$, by Gordina \cite{Gordina2000} for $\SO(N)$, by Chan \cite{Chan20} for classical compact groups, and by Olafsson and Wiboonton \cite{OW2013} for symmetric spaces. More precisely, let $\{G^N\}_{N=1}^\infty$ be a family of compact matrix Lie groups or symmetric spaces.  One then lets $N$ tend to infinity and observe if the Segal--Bargmann transform on $G^N$ tends to some linear transformation.
	
	\smallskip
	
	In case $\{G^N\}_{N=1}^\infty$ is a family of symmetric spaces, the work of Olafsson and Wiboonton \cite{OW2013} considers these spaces as a sequence of nested manifolds. In case of the spheres, this would mean $G^N=S^N$ is the \emph{unit} $N$-dimensional sphere and $S^{N-1}$ is a submanifold of $S^N$ in the usual way. The authors then describe the large-$N$ limit of the Segal--Bargmann transforms on $G^N$'s using a direct limit construction. However, there is no geometric model for the limiting transform with this construction. The difficulty can be traced to the fact that the volume measure on the \emph{unit} sphere does not have a natural limit as the dimension tends to infinity.
	
	In this work, we approach the problem of large-$N$ limit on the spheres from a different angle. Instead of working with the family of unit spheres, we consider the spheres $S^{N-1}(\sqrt N)$ of radius $\sqrt N$. In particular, we do not consider them as a nested sequence of manifolds. Umemura and Kono \cite{UmKo65} and Petersen and Sengupta \cite{PeSen} observed that this consideration leads to a desirable scaling of the spherical Laplacian $\Del_{S^{N-1}(\sqrt N)}$, which has a limit that is the \emph{Hermite differential operator} as $N$ tends to infinity. In addition, it is a very old observation from statistical mechanics that the \emph{normalized} volume measure $\ovl\sig^{N-1}$ of $S^{N-1}(\sqrt N)$ converges to the standard Gaussian measure $\mu^\infty_1$ on $\R^\infty$ (\cites{Boltzmann,UmKo65,PeSen}). Using these results, we formulate the large-$N$ limit of Segal--Bargmann transform on the sphere $S^{N-1}(\sqrt N)$ as the exponential of the Hermite differential operator and explicitly construct its range. Our result not only is different from that of \cite{OW2013}, but it also provides us with a geometric interpretation of both the limiting transform and the limiting range. Our approach is also desirable because it applies to all polynomials, which are dense in $L^2\big(S^{N-1}(\sqrt N),\ovl\sig^{N-1}\big)$, as compared to the result for $\U(N)$ in \cite{Biane97} and \cite{DHK2013} working only with class functions.
	
	\smallskip
	
	The structure of this paper is as follows. In Sect.~\ref{chap:summary}, we provide some background results that help to formulate the Main Theorem (Thm.~\ref{thm:main}). In Sect.~\ref{chap:prelim}, we discuss the two-parameter Segal--Bargman transform on the Euclidean space $\R^k$, denoted by $\cal B^k_{s,t}$ (Thm.~\ref{thm:heat-real-general}) and its action on polynomials (Thm.~\ref{thm:heat-powerseries}). In Sect.~\ref{chap:coherent-states-sphere}, we review important properties of the spherical Laplacian $\Del_{S^{N-1}(\sqrt N)}$ and the \emph{invariant Segal--Bargmann transform} acting on $L^2\big(S^{N-1}(\sqrt N),\ovl\sig^{N-1}\big)$ (Thm.~\ref{thm:SB-sphere}). In Sect.~\ref{sec:large-N-sphere}, we remind the reader of the large-$N$ limit of $\Del_{S^{N-1}(\sqrt N)}$ and the spherical measure $\ovl \sig^{N-1}(\sqrt N)$ (Thm.~\ref{thm:limitsph} and Prop.~\ref{prop:Sen-Hermite}). In Sect.~\ref{chap:proof}, we provide the proof of the Main Theorem \ref{thm:main} by considering separately the limit of the range of the Segal--Bargmann transform (Thm.~\ref{thm:limitquad}), and the limit of the transform itself (Thms.~\ref{thm:horizontal_arrow} and \ref{thm:limitmap}).
	
%===================================
\section{Statements of Main Results}\label{chap:summary}
%===================================
	
	Consider the following $(N-1)$-sphere of radius $\sqrt N$ centered at the origin of $\R^N$,
	\begin{align*}
		S^{N-1}(\sqrt N)=\left\{\mb x=(x_1,\dots,x_N)\in\R^N\colon x_1^2+\dots+x_N^2=N\right\},
	\end{align*}
	and the space of square-integrable functions on $S^{N-1}(\sqrt N)$with respect to the normalized spherical volume measure $\ovl\sig^{N-1}$, denoted by $L^2(S^{N-1}(\sqrt N),\ovl\sig^{N-1})$. We want to study the behavior of the Segal--Bargmann transform on $L^2(S^{N-1}(\sqrt N),\ovl\sig^{N-1})$ when $N$ is large. There are three main ingredients in the problem. 
	\begin{enumerate}[(I)]
		\item 	The scaling of the radius by $\sqrt{N}$ is important for the measures $\ovl \sig^{N-1}$ to have some notion of convergence. It is a well-known result (see \cite{PeSen} and the references therein) that with this scaling, if one consider $S^{N-1}(\sqrt N)\subset\R^N$ as embedded in the space of infinite real sequences $\R^\infty$, then the measure $\ovl{\sig}^{N-1}$ converges to $\mu^\infty_1$, the infinite product of the standard Gaussian measure $\mu^1_1(dx)=(2\pi)^{-\frac12} e^{-\frac{x^2}2}\,dx$ on $\R$. (The superscript stands for the dimension and the subscript $1$ stands for unit variance).
		
		Specifically, let $p$ be a polynomial of fixed $k$ variables $x_1,\dots, x_k$, considered as a function both on $S^{N-1}(\sqrt N)$ ($N\ge k$) and on $\R^\infty$. Then 
		\begin{align*}
			\lim_{N\to\infty} \int_{S^{N-1}(\sqrt N)} p\,d\ovl\sig^{N-1}=\int_{\R^\infty} p\,d\mu^\infty_1.
		\end{align*}
		\item As $N\to\infty$, the spherical Laplacian $\Del_{S^{N-1}(\sqrt N)}$ tends to the differential Hermite operator ${\bf H}$ given by
		\begin{align*}
			{\mb H} =\sum_{j=1}^\infty \frac{\pt^2}{\pt x_j^2}-\sum_{j=1}^\infty x_j\frac {\pt}{\pt x_j}. 
		\end{align*}
		This means for any polynomial $p$, considered as a function both on the spheres and on $\R^\infty$ as in (I), we have
		\begin{align*}
			\lim_{N\to\infty} \Del_{S^{N-1}(\sqrt N)}\,p= {\mb H}\;p.
		\end{align*}
		\item The Segal--Bargmann transform (of parameter $T>0$) on the sphere $S^{N-1}(b)$ of radius $b>0$, with an explicit construction in \cite{HallMit2002}, is a unitary map from $L^2(S^{N-1}(b))$ onto some Hilbert space of holomorphic functions on the quadric defined via the heat operator $e^{\frac{T}{2}\Del_{S^{N-1}(b)}}$ induced by the Laplacian $\Del_{S^{N-1}(b)}$.   	
	\end{enumerate}
	
	%-----------------------------------------------------------------------------
	\subsection{Segal--Bargmann transforms on $S^{N-1}(\sqrt N)$}	
	%-----------------------------------------------------------------------------
	Without introducing too many technical points, we give a brief description of the Segal--Bargmann transforms on the spheres $S^{N-1}(\sqrt N)$  to state the Main Theorem. More details of the formulations can be found in Sect.~\ref{chap:coherent-states-sphere}.
	
	\medskip
	
	%Let $L^2(S^{N-1}(\sqrt N),\ovl{\sig}^{N-1})$ be the space of square integrable functions on $S^{N-1}(\sqrt N)$ with respect to the normalized rotational invariant volume measure $\ovl\sig^{N-1}$. 

	The quadric 
	\begin{align*}
		Q^{N-1}(\sqrt N)=\{\mb a=(a_1,\dots,a_N)\in\C^N\colon \mb a\cdot \mb a=a_1^2+\dots+a_N^2=N\}
	\end{align*}
	is the complexification of $S^{N-1}(\sqrt N)$.
	
	For any $T>0$, there is a measure $\nu^{N-1}_T$ on the quadric $Q^{N-1}(\sqrt N)$ given by the fundamental solution of the heat equation and the volume measure of the $(N-1)$-hyperbolic space of sectional curvature $-\frac{1}{N}$. Denote by  $\cal HL^2\big(Q^{N-1}(\sqrt N),\nu^{N-1}_{T}\big)$ the Hilbert space of \emph{holomorphic} functions on $Q^{N-1}(\sqrt N)$ that are square-integrable with respect to $\nu^{N-1}_T$.

	The Segal--Bargmann transform $\cal C^{N-1}_{T}$ with time $T>0$ on the sphere $S^{N-1}(\sqrt N)$ is a unitary map from $L^2\big(S^{N-1}(\sqrt N),\ovl{\sig}^{N-1}\big)$ onto $\cal HL^2\big(Q^{N-1}(\sqrt N),\nu^{N-1}_{T}\big)$. It is given by the composition of the \emph{heat operator} $e^{\frac{T}{2}\Del_{S^{N-1}(\sqrt N)}}$ with the holomorphic extension operation. That is,
	\begin{align*}
		\cal C^{N-1}_{T}f = \left(e^{\frac{T}{2}\Del_{S^{N-1}(\sqrt N)}}f\right)_\C,\quad f\in L^2(S^{N-1}(\sqrt N),\ovl{\sig}^{N-1}),
	\end{align*}
	where $(\cdot)_\C$ denotes the holomorphic extenstion of a sufficiently nice function from $S^{N-1}(\sqrt N)$ to $Q^{N-1}(\sqrt N)$.In particular, one has 
	\begin{align*}
	\|f\|_{L^2(S^{N-1}(\sqrt N),\ovl \sig^{N-1})}
		=\left\|\cal C^{N-1}_T\,f\right\|_{\cal HL^2(Q^{N-1}(\sqrt N),\nu^{N-1}_T)}.
	\end{align*}
	
	When $f$ is a polynomial on $\R^N$ considered as a function on the sphere $S^{N-1}(\sqrt N)$, the heat operator can be computed as a power series expansion of the exponential operator:
	\begin{align*}
		e^{\frac{T}{2}\Del_{S^{N-1}(\sqrt N)}}f = \sum_{n=0}^\infty \frac{1}{n!} \left(\frac{T}{2}\Del_{S^{N-1}(\sqrt N)}\right)^n f.
	\end{align*}
	
	%------------------------
	\subsection{Main results}\label{sec:main-res}
	%------------------------
	Briefly speaking, when one lets $N\to\infty$, there are three convergence phenomena. Firstly, the measure $\ovl\sig^{N-1}$ converges to the Gaussian $\mu_1^\infty$ in the sense of Ingredient (I). Secondly, the measure $\nu^{N-1}_T$ also converges to some Gaussian measure defined on the space of complex sequences $\C^\infty$. Finally, since $\Del_{S^{N-1}(\sqrt N)}$ converges to the Hermite operator $\mb H$, the heat operator $e^{\frac{T}{2}\Del_{S^{N-1}(\sqrt N)}}$ converges to the exponential operator $e^{\frac T2 \mb H}$. 
	
	\bigskip 
	
	More rigorously, for any $T>0$, let $\ga^1_T$ be the Gaussian measure on $\C^1$ given by 
	\begin{align}\label{eqn:meas-gamma}
		\ga^1_T(du+i\,dv)= (\pi^2(e^{2T}-1))^{-\frac 12}e^{-\frac{u^2}{e^T+1}-\frac{v^2}{e^T-1}}\,du\,dv,\quad u,v\in\R.
	\end{align}
	%and $\cal HL^2(\C^k,\ga^k_t)$ be the space of holomorphic square-integrable functions with respect to $\ga^k_t$. 
	Let $\ga^\infty_T$ denote the infinite product of the measure $\ga^1_T$, and let the space $\cal H L^2(\C^\infty,\ga^\infty_T)$ denote the closure of the set of \emph{holomorphic polynomials} with respect to $\ga^\infty_T$.
	
	\medskip 
	
	We want to establish the following theorem for the remainder of this work.	
	\begin{thm}\label{thm:main} Fix some $T>0$.
		\begin{enumerate} [(1)]
			\item There exists a unique map $\mb B_T$ acting on $L^2(\R^\infty,\mu^\infty_1)$ determined by
			\begin{align*}
				\mb B_T\,q=\left(e^{\frac{T}{2}\mb H}\,q\right)_\C,\qquad q\text{ is a polynomial of finitely many variables},
			\end{align*}
			such that $\mb B_T$ is a unitary map from $L^2(\R^\infty,\mu^\infty_1)$ onto $\cal HL^2(\C^\infty,\ga^\infty_T)$. 
			\item The measure $\nu^{N-1}_T$ converges to the infinite product measure $\ga^\infty_k$ in the following sense: for any polynomial $q$ of fixed $2k$ variables $a_1,\dots, a_k, \ovl a_1,\dots,\ovl a_k$ considered as a function both on $Q^{N-1}(\sqrt N)$, with $N\ge k$, and on $\C^\infty$, we have
			\begin{align*}
				\lim_{N\to \infty} \int_{Q^{N-1}(\sqrt N)}q\,d\nu^{N-1}_T = \int_{\C^\infty} q\,d\ga^\infty_T.%=\int_{\C^k} q\,d\ga^k_T.
			\end{align*}
			\item The Segal--Bargmann transform $\cal C^{N-1}_T$ converges to the map $\mb B_T$ as $N\to\infty$ in the following sense: for any polynomial $q$ of fixed $k$ variables $x_1,\dots, x_k$ considered as a function both on the spheres $S^{N-1}(\sqrt N)$, with $N\ge k$, and on $\R^\infty$, we have 
			\begin{align*}
				\lim_{N\to \infty} \cal C^{N-1}_T\,q = \mb B_T\,q.
			\end{align*}
			In particular,
			\begin{align*}
				\lim_{N\to\infty}\|q\|_{L^2(S^{N-1}(\sqrt N),\ovl \sig^{N-1})}
				=\lim_{N\to \infty} \left\|\cal C^{N-1}_T\,q\right\|_{\cal HL^2(Q^{N-1}(\sqrt N),\nu^{N-1}_T)}
				= \left\|\mb B_T\,q\right\|_{\cal HL^2(\C^\infty,\ga^\infty_T)}.
			\end{align*} 
		\end{enumerate}
	\end{thm}
	
	\medskip
	
	One can understand the theorem using the following commutative diagram:
	\begin{center}
		\begin{tikzcd}[row sep=huge, column sep=10em]
			L^2\big(S^{N-1}(\sqrt{N}),\ovl \sig^{N-1}\big)\arrow[r, "\tiny{\left(e^{\frac{T}{2}\Del_{S^{N-1}(\sqrt N)}}\ \underline{\hspace{1em}}\right)_\C}"] \arrow[d, dashed, "N\to \infty"]
			& \cal H L^2\big(Q^{N-1}(\sqrt{N}),\nu^{N-1}_{T}\big) \arrow[d, dashed, "N\to\infty"] \\
			L^2(\R^\infty,\mu^\infty_1) \arrow[r, "\tiny{\left(e^{\frac{T}{2}\mb H}\ \underline{\hspace{1em}}\right)_\C}"]
			&  \cal H L^2(\C^\infty,\ga^\infty_T) 
		\end{tikzcd}
	\end{center}
	where each of the horizontal maps is a unitary map from the $L^2$-space on the left-hand side onto the holomorphic $L^2$-space on the right-hand side, and the vertical limits are limits of measures  understood in the sense of polynomials as in Thm.~\ref{thm:main} (2) and (3). 	
	
%====================================================================
\section{Summary of Segal--Bargmann Transforms on Euclidean Spaces}\label{chap:prelim}
%====================================================================
In the sequel, to be consistent with the choice of notations in \cite{HallMit2002}, we will denote complex variables (especially the quadric variables) by $a_1,a_2,\dots$ instead of $z_1,z_2,\dots$. We denote by $\N$ the set of natural numbers (including 0) and by $\N^*$ the set of positive integers. For two complex vectors complex vectors $\mb a=(a_1,\dots,a_k),\mb b=(b_1,\dots, b_k)$ in $\C^k$, we denote by $\mb a\cdot \mb b$ the \emph{real} dot product $\sum_{j=1}^ka_jb_j$. We also write $\mb a^2$ for the dot product $\mb a\cdot \mb a$. Note that the $\C^k$-norm of the vector $\mb a$ satisfies $|\mb a|^2=\mb a\cdot\ovl{\mb a}$.

\medskip

Let us first look at the Segal--Bargmann transform on the space of square-integrable functions on $\R^k$ for any $k\in\N^*$. 
%-----------------------------------------------------------------
\subsection{Two-parameter Segal--Bargmann transforms}\label{sec:two-parameter-euclidean}
%-----------------------------------------------------------------
The Gaussian ${\mu^k_t(\mb x)=(2\pi t)^{-\frac k2}e^{-\frac{\mb x^2}{2t}}}$ is the fundamental solution to the \emph{heat equation}
\begin{align}
	\frac{\pt}{\pt t} K(\mb x,t)=\frac12 \Del_{\R^k} K(\mb x,t),\qquad \mb x\in\R^k, t>0\label{eqn:heat-real}.
\end{align}
%i.e., $K(\mb x,t)=\mu^k_t(\mb x)$ solves the differential equation \eqref{eqn:heat-real} with the initial condition
%\begin{align*}
%	\lim_{t\downarrow 0} \int_{\R^k} f(\mb x)\,d\mu^k_t(\mb x)=f(\mb 0).
%\end{align*}
%for all $f\in C^\infty_c(\R^k)$. 

The \emph{heat-kernel measure} on $\R^k$ given by 
\begin{align*}
	\mu^k_t(d\mb x)=\mu^k_t(\mb x)\,d\mb x=(2\pi t)^{-\frac k2}e^{-\frac{\mb x^2}{2t}}\,d\mb x 
\end{align*}
is a probability measure with mean $0$ and variance $t$, which plays an important role both in theory of differential equations and in probability. The \emph{heat operator} $e^{\frac t2\Del_{\R^k}}$ applying to a function $f\colon \R^k\to \C$ is the convolution
\begin{align}
	\left(e^{\frac t2 \Del_{\R^k}} f\right)(\mb x):=(f*\mu^k_t)(\mb x)=\int_{\R^k} f(\mb y)\mu^k_t(\mb x-\mb y)\,d\mb y
	=\int_{\R^k} f(\mb x-\mb y)\mu^k_t(\mb y)\,d\mb y,
	\label{eqn:heat-inhomogeneous}
\end{align}
whenever this integral is convergent for all $\mb x\in \R^k$.

Let $\cal H(\C^k)$ denote the space of entire holomorphic functions on $\C^k$, and let $\cal HL^2(\C^k,m)$ or $\cal HL^2(\C^k,dm)$ denote the set of holomorphic functions on $\C^k$ that are square integrable with respect to the measure $m$. 

There have been many studies on Segal--Bargmann transforms in the Euclidean case. Notable results in this topic concern when the heat operators are isometries from $L^2(\R^k,\mu^k_t)$ into some Hilbert spaces that are subspaces of $\cal H(\C^k)$ (see {\cite{Hall97,Hall2000}}). In this work, we are also interested in a form of the Segal--Bargmann transform on $\R^k$ where the variance of the Gaussian defining the domain $L^2$-space is different from the time of the heat operator. The equivalent version for compact Lie groups is given in \cite[Sect.~3]{Hall99}.

\begin{thm}
	\label{thm:heat-real-general} For $k\in\N^*$, real numbers $s,t$ with $0<t<2s$, and real vectors $\mb u,\mb v\in\R^k$, define the following probability measure on $\C^k=\R^{2k}$:
	\begin{align*}
		\xi^k_{s,t}(d\mb u+i\,d\mb v)=\xi^k_{s,t}(\mb u+i\mb v)\,d\mb u\,d\mb v:=(\pi (2s-t))^{-\frac k2}(\pi t)^{-\frac k2}e^{-\frac{\mb u^2}{2s-t}}e^{-\frac{\mb v^2}{t}}\,d\mb u\,d\mb v.
	\end{align*}	
	Define the \emph{Segal--Bargmann transform $\cal B^k_{s,t}$} acting on the Hilbert space $L^2(\R^k,\mu^k_s)$ as follows:  for any $f$ in $L^2(\R^k,\mu^k_s)$, the function $\cal B^k_{s,t}f$ is given by
	\begin{align}\label{eqn:heat-real-transform}
		\left(\cal B_{s,t}^k\,f\right)(\mb z):=\left(e^{\frac t2\Del_{\R^k}}f\right)_\C(\mb z)=(2\pi t)^{-\frac k2}\int_{\R^k}f(\mb y)e^{-\frac{(\mb z-\mb y)^2}{2t}}\,d\mb y,\quad \mb z\in\C^k,
	\end{align}
	where $F_\C$ denotes the holomorphic extension of a function $F\colon \R^k\to \C$ to $\C^k$.
	Then
	\begin{enumerate}[(i)]
		\item The set $\cal HL^2(\C^k,\xi^k_{s,t})$ is a Hilbert subspace of $L^2(\C^k,\xi^k_{s,t})$.
		\item For any $f$ in  $L^2(\R^k,\mu^k_s)$, the function $\cal B_{s,t}^k\,f$ is entire holomorphic, and $\cal B^k_{s,t}$ is a unitary map from $L^2(\R^k,\mu^k_s)$ onto $\cal HL^2(\C^k,\xi^k_{s,t})$.
	\end{enumerate}
\end{thm}
Note that elements in the space $\cal H L^2(\C^k,\xi^k_{s,t})$ are actual holomorphic functions and not just equivalence classes of almost-everywhere defined functions.

\begin{rem}	\label{rem:heat-operator-real}
		We have for each fixed $\mb x$
		\begin{align*}
			\lim_{s\to \infty}  (2\pi s)^{\frac k2}\mu^k_{s}(\mb x)=\lim_{s\to\infty} e^{-\frac{\mb x^2}{2s}}=1,
		\end{align*}
		so one can interpret the measure $(2\pi s)^{\frac k2}\,d\mu^k_{s}$ when $s\to\infty$ as the Lebesgue measure $d\mb x$ on $\R^k$, while the measure $(2\pi s)^{\frac k2}\,d\xi^k_{s,t}(\mb u,\mb v)$ tends to the product measure $d\mb u\,d\mu^k_{t/2}(\mb v)$. The Hilbert space $\cal HL^2\big(\C^k,d\mb u\,d\mu^k_{t/2}(\mb v)\big)$ is the set of entire holomorphic functions $F$ where 
		\begin{align*}
			\|F\|^2_{\cal HL^2(\C^k,d\mb u\,d\mu^k_{t/2}(\mb v))}=(\pi t)^{-\frac{k}{2}}\int_{\R^{2k}} |F(\mb u+i\mb v)|^2e^{-\frac{\mb v^2}{t}}\,d\mb u\,d\mb v<\infty.
		\end{align*}
		The Segal--Bargmann transform defined on the right-hand side of \eqref{eqn:heat-real-transform} is then a unitary map from $L^2(\R^k,d\mb x)$ onto $\cal HL^2(\C^k,\mu^k_{t/2}(\mb v)\,d\mb v\,d\mb u)$. This is known as the ``$C$-version" or the \emph{invariant} Segal--Bargmann transform  in \cite{Hall2000}.
\end{rem}

%----------------------------------------------------------------------------
\subsection{Actions of Segal--Bargmann transforms on polynomials} \label{sec:action-polynomials}
%-----------------------------------------------------------------------------
Let us introduce the following notations:
\begin{enumerate}[(i)]
	\item $\cal P_k(\R)$: the set of polynomials of $k$ real variables $x_1,\dots, x_k$,
	\item $\cal P_k(\C)$: the set of polynomials of $2k$ complex variables $a_1,\dots, a_k$, $\bar a_1,\dots,\bar a_k$,
	\item $\cal P^{\le l}_k(\C)$: the set of polynomials in $\cal P_k(\C)$ of degree at most $l$,
	\item $\cal P^{\le l}_k(\R):=\cal P_k(\R)\cap \cal P^{\le l}_k(\C)$,
	\item $\cal H\cal P_k(\C):=\cal H(\C^k)\cap \cal P_k(\C)$,
	\item $\cal H\cal P^{\le l}_k(\C):=\cal H(\C^k)\cap \cal P^{\le l}_k(\C)$.
\end{enumerate}
We allow all polynomials (of real or complex variables) to have complex coefficients. We have several important remarks that will be assumed from now on:
\begin{rem}\label{rem:polynomials}
	\begin{enumerate}[(1)]
		\item For all $k\in\N^*$, we identify $\R^k$ as a subspace of $\R^\infty$ and identify $\cal P_{k}(\F)$ as a subspace of $\cal P_{N}(\F)$ for all $k\le N$, where $\F=\R$ or $\F=\C$, in the obvious way. 		
		
		\item For a fixed $k$ let $\pt_r=\frac{\pt}{\pt r}$ denote the partial derivative with respect to the radius $r=\sqrt{x_1^2+\dots+x_k^2}$ of a real function written in polar coordinates. Then, by chain rule
		\begin{align*}
			(r\pt_r)_k:=r\frac{\pt}{\pt r}=\sum_{j=1}^k x_j\frac{\pt}{\pt x_j}.
		\end{align*}
		This is the well-known \emph{Euler identity}, and the rightmost differential operator is the \emph{Cauchy--Euler operator}. 
		
		For any $N\ge k$, one can easily see that the two operators $\Del_{\R^N}$ and $(r\pt_r)_N$ actually do not depend on $N$ when acting on a function $f\in C^2(\R^k)$ of $k$ variables $x_1,\dots, x_k$ considered as a function on $\R^N$, in particular, any element of $\cal P_k(\R)$. More explicitly, 
		\begin{align*}
			\Del_{\R^k} f=\Del_{\R^N} f\quad\tnm{and}\quad (r\pt_r)_k f=(r\pt_r)_N f.
		\end{align*} 
	
		\item Similarly, the Gaussian measures $\mu^N_t$ are independent of $N$ if we are integrating a function $f$ of fixed $k\le N$ variables with respect to these measures, i.e., if $f\in L^1(\C^k,\mu^k_t)$ is a function of $k$ variables $x_1,\dots, x_k$ considered also as a function of $N$ variables, then 
		\begin{align*}
			\int_{\R^N} f(\mb y)\,d\mu^N_t(\mb y)=\int_{\R^k} f(\mb x)\,d\mu^k_t(\mb x).
		\end{align*}
		
	\end{enumerate}
\end{rem}	

In connection with Segal--Bargmann transforms, let us state the following well-known result (see \cite{Rob91}).
\begin{thm}\label{thm:heat-powerseries}
	For any $t>0$ and $q\in\cal P_k(\R)$ considered as a function of $N$ variables ($N\ge k$), we have
	\begin{align*}
		e^{\frac{t}2\Del_{\R^N}} q=\sum_{n=0}^\infty \frac{1}{n!}\left(\frac{t}{2}\Del_{\R^N}\right)^n q,
	\end{align*}
	at all points in $\R^N$, where the operator on the left-hand side is the heat operator defined in \eqref{eqn:heat-inhomogeneous}, and the power series of differential operators on the right-hand side applies term-wise to $q$.
\end{thm}
Thm.~\ref{thm:heat-powerseries} justifies the use of the notation $e^{\frac{t}{2}\Del_{\R^N}}$ for the heat operator: the action of the heat operator on polynomials is nothing but the Taylor expansion of the exponential operator of the Laplacian. 

%======================================================================
\section{Some Results on the Segal--Bargmann transforms on the Spheres}\label{chap:coherent-states-sphere}
%======================================================================

In this section, we recall some important facts about the Segal--Bargmann transforms on the spheres following the construction of Stenzel \cite{Stenzel99} and Hall and Mitchell \cite{HallMit2002}, as well as the large-$N$ limit phenomena of the spherical measures and the spherical Laplacians by Petersen and Sengupta \cite{PeSen} and Umemura and Kono \cite{UmKo65}.

Since the results in \cite{HallMit2002} are established for the unit spheres rather than $S^{N-1}(\sqrt N)$, we will adjust these results to an arbitrary radius $b>0$.

%----------------------------------------------------------------
\subsection{The spherical Laplacian and the spherical heat kernel}\label{sec:heat-sphere}
%----------------------------------------------------------------
For any $b>0$, consider the $(N-1)$-sphere of radius $b$ centered at the origin of $\R^N$,
\begin{align*}
	S^{N-1}(b)=\left\{\mb x=(x_1,\dots,x_N)\in\R^N\colon x_1^2+\dots+x_N^2=b^2\right\}.
\end{align*}
Let $\ovl\sig^{N-1}_b$ be the rotation-invariant volume measure of $S^{N-1}(b)$. 

The spherical Laplacian is well-studied, but we only need two equivalent formulas, given in the following result.

\begin{prop}\label{prop:spherical-laplacian}
	The Laplacian on the sphere $S^{N-1}(b)$ is given by
	\begin{align}
		\Del_{S^{N-1}(b)}
		&= \frac{1}{b^2}\left(r^2\Del_{\R^N}-(r\pt_r)^2_N-(N-2)(r\pt_r)_N\right)\big|_{S^{N-1}(b)} \label{eqn:sphere-Laplacian-1}\\
		&=\frac{1}{b^2}\sum_{1\le k<l\le N} \left.\left(x_k\frac{\pt}{\pt x_l}-x_l\frac{\pt}{\pt x_k}\right)^2\right|_{S^{N-1}(b)} \label{eqn:sphere-Laplacian-2},
	\end{align}
	where $(r\pt_r)_N$ is as in Rem.~\ref{rem:polynomials}.
\end{prop}
Formula \eqref{eqn:sphere-Laplacian-1} comes from the Euclidean Laplacian written in polar coordinates (see, e.g., \cite[Chap.~3 Exerc.~1]{Jost2017}), while formula \eqref{eqn:sphere-Laplacian-2} comes from the computation of Laplace--Beltrami operator on the sphere using angular momenta (see, e.g., \cite[Sect.~3.6]{PeSen}). The two formulas should be understood as follows: to compute the spherical Laplacian of a $C^2$-function $f$ on $S^{N-1}(b)$, we first extend $f$ to a $C^2$-function on a neighborhood of $S^{N-1}(b)\subset \R^N$, then apply the differential operator given in either  \eqref{eqn:sphere-Laplacian-1} or \eqref{eqn:sphere-Laplacian-2}, and finally restrict back to the sphere. The final result is independent of the choice of extension.

\smallskip

The \emph{heat kernel} $\rho^{N-1}_{b,t,\mb x}(\mb y)$ at a base point $\mb x$, time $t$, and a position $\mb y$ is the fundamental solution to the \emph{spherical heat equation}
\begin{align*}
	\frac{\pt}{\pt t}\, K(t,\mb x,\mb y)=\frac{1}{2}\Del_{S^{N-1}(b),\mb y}\, K(t,\mb x,\mb y),
\end{align*}
with the initial condition
\begin{align*}
	\lim_{t\downarrow 0} \int_{\mb y\in S^{N-1}(b)} \rho^{N-1}_{b,t,\mb x}(\mb y)f(\mb y)\,d\ovl\sig^{N-1}_b(\mb y)=f(\mb x),\qquad f\in C^\infty(S^{N-1}(b)).
\end{align*}
Here, the notation $\Del_{S^{N-1}(b),\mb y}$ is the spherical Laplacian of the variable $\mb y$. 
It is known that $\rho^{N-1}_{b,t,\mb x}(\mb y)$ is symmetric in $\mb x$ and $\mb y$ and only depends on the (spherical) distance between $\mb x$ and $\mb y$.

%----------------------------------------------------
\subsection{The hyperbolic space and its heat kernel}
%----------------------------------------------------
An important element in Stenzel's construction is the heat kernel on the noncompact dual symmetric space. It is known that the dual space of the sphere $S^{N-1}(b)$ is the hyperbolic space $H^{N-1}(b)$ of sectional curvature $-\frac{1}{b}$. 

By Cartan--Hadamard Theorem, $H^{N-1}(b)$ is diffeomorphic to $\R^{N-1}$ via the exponential map, so for each $\mb x\in S^{N-1}(b)$ we can identify each tangent space $T_{\mb x}S^{N-1}(b)$ with $H^{N-1}(b)$. Let $p$ be the hyperbolic distance from a point $\mb p\in H^{N-1}(b)$ to the base point $\mb x$. 
\begin{prop}\label{prop:hyperbol-inv-measure}
	The measure $\be^{N-1}_b(p)\,d\mb p$ given by
	\begin{align*}
		\be^{N-1}_b(p)=
			\displaystyle\left(\dfrac{b}{p}\sinh\left(\dfrac{p}{b}\right)\right)^{N-2}
	\end{align*}
	is a volume measure on the hyperbolic space $H^{N-1}(b)$.
\end{prop}
%Let $G_{N-1}$ be the group of all orientation-preserving isometries on $H^{N-1}(b)$. Since the Laplacian commutes with the action of $G_{N-1}$ on $H^{N-1}(b)$, if $K(s,\mb x,\mb y)$ is the heat kernel on the hyperbolic space with base point $x$ at time $s$ and point $\mb y$, then $K(t,\mb x,\mb y)=K(t,g\mb x,g\mb y)$ for all $g\in G_{N-1}$. Hence, the kernel actually depends only on $t$ and $p(\mb y)=d(\mb x,\mb y)$. Re-denote 
The heat kernel of $H^{N-1}(b)$ is denoted by $\theta^{N-1}_{b}(s,p)$. The heat kernel satisfies the \emph{hyperbolic heat equation}
\begin{align*}
	\frac{\pt}{\pt s}\theta^{N-1}_{b}(s,p)=\frac12 \left[\frac{\pt^2}{\pt p^2}+\dfrac{N-2}{b}\coth \left(\dfrac {p}{b}\right)\frac{\pt}{\pt p}\right]\theta^{N-1}_{b}(s,p).
\end{align*}
%(There is a typographical error of a minus sign in \cite[Thm.~2]{HallMit2002}.)
	
The right-hand side of the heat equation (ignoring the constant $\frac {1}{2}$) is the \emph{radial part} of the Laplacian $\Del_{H^{N-1}(b)}$. We can check that $\theta^{N-1}_{b}(s,p)=\theta^{N-1}_{1}(\frac{s}{b^2},\frac {p}{b})$. 

The following estimate in Davies {\cite[Thm.~5.7.2]{Davies}} on the growth of the heat kernel will be useful later. Note that here our time parameter is different from Davies' by a factor of 2.
\begin{prop}\label{prop:estimate-heat-hyperbolic}
	For any $d\in\N$ and $d>1$, there exists a constant $c_d>0$ such that the hyperbolic heat kernel satisfies
	\begin{align*}
		c_{d}^{-1} K_{d}(s,p)\le \theta^{d}_{1}(s,p) \le c_{d} K_{d}(s,p)
	\end{align*}
	\nopagebreak[4] where 
	\begin{align*}
		K_{d}(s,p)= (2\pi s)^{-\frac{d}{2}} (1+r)\left(1+p+\frac s2\right)^{\frac {d-1}{2}} e^{-\frac{s(d-1)^2}{8}-\frac{p(d-1)}{2}}e^{-\frac{p^2}{2s}}. 
	\end{align*}
	In particular, there exists a positive function $k\colon \N\times \R^+\to\R^+$ such that 
	\begin{align*}
		K_d(s,p)\le k(d,s)\, e^{-\frac{p^2}{2s}}.
	\end{align*}
\end{prop}
%------------------------------------------
\subsection{The Segal--Bargmann transforms}\label{sec:SB-sphere}
%------------------------------------------
We cite some notable results established by Hall and Mitchell \cite{HallMit2002}. Recall that $Q^{N-1}(b)$ denotes the \emph{complexification of the sphere} $S^{N-1}(b)$ or the \emph{quadric} of parameter $b>0$, i.e.,
\begin{align}
	Q^{N-1}(b) = \left\{(a_1,\dots,a_n)\in\C^N\colon \mb a^2=a_1^2+\dots+a_n^2=b^2\right\}.\label{eqn:quadric-def}
\end{align}
One should distinguish the quadric, which is a (non-compact) complex submanifold of $\C^N$, from the complex sphere $S^{2N-1}(b)=\{|a_1|^2+\dots+|a_n|^2=b^2\}$.

\begin{prop}[{\cite[Sect.~III]{HallMit2002}}]\label{prop:quadric-diffeomorphic-cotang}
	Let $b=|\mb x|$ be the parameter of the quadric and $p=|\mb p|$.  The map
	\begin{align*}
		\mb a(\mb x,\mb p)=\mb u(\mb x,\mb p)+i\mb v(\mb x,\mb p):=
		\begin{cases}
			\displaystyle\cosh\left(\dfrac{p}{b}\right)\,\mb x+\frac{b}{p}\sinh\left(\dfrac{p}{b}\right)\,i\mb p\quad&\tnm{if $p>0$},\\
			\mb x+i\mb 0\quad&\tnm{if $p=0$}			
		\end{cases}
	\end{align*}
	defines a diffeomorphism from $TS^{N-1}(b)=\{(\mb x,\mb p)\in S^{N-1}(b)\times \R^N\colon \mb x \cdot \mb p=0\}$ to $Q^{N-1}(b)$. 
\end{prop}

Note that $\mb a^2=\mb a\cdot \mb a=b^2$ is constant but $|\mb a|^2=\mb a\cdot \ovl {\mb a}$ is unbounded. Similar to the Euclidean Segal--Bargmann transforms, we have the following results.
\begin{prop}
	For any $b,t>0$ and $N>1$, to each function $f$ that belongs to the Hilbert space $L^2\big(S^{N-1}(b),\ovl\sig^{N-1}_b\big)$, we associate a function $e^{\frac t2\Del_{S^{N-1}(b)}}f$ given by
	\begin{align}
		\left(e^{\frac t2\Del_{S^{N-1}(b)}}f\right)(\mb x):=\int_{\mb y \in S^{N-1}(b)}\rho^{N-1}_{b,t,\mb x}(\mb y)f(\mb y)\,d\ovl\sig^{N-1}_b(\mb y),\quad \mb x\in S^{N-1}(b).\label{def:Segal-Bargmann sphere}
	\end{align}
	The map $e^{\frac t2\Del_{S^{N-1}(b)}}$ is well defined for all $f\in L^2\big(S^{N-1}(b),\ovl\sig^{N-1}_b\big)$, $\mb x\in S^{N-1}(b)$, and is called the \emph{heat operator} at time $t$.
\end{prop}
\begin{defi}\label{def:SB-sphere}
	For any $b,t>0$ and $N>1$, we define the \emph{Segal--Bargman transform} $\cal C^{N-1}_{b,t}$ on $L^2\big(S^{N-1}(b),\ovl\sig^{N-1}_b\big)$ as follows: for any $f\in L^2\big(S^{N-1}(b),\ovl\sig^{N-1}_b\big)$, $\cal C^{N-1}_{b,t}f$ is a function on the quadric $Q^{N-1}(b)$ given by
	\begin{align*}
		\left(\cal C^{N-1}_{b,t}\,f\right)(\mb a)&=\left(e^{\frac t2\Del_{S^{N-1}(b)}}f\right)_\C(\mb a), \qquad \mb a\in Q^{N-1}(b),
	\end{align*}
	whenever the integral operator is defined. Here, $F_\C$ denotes the holomorphic extension of a sufficiently nice function $F\colon S^{N-1}(b)\to \C$ to the whole $Q^{N-1}(b)$.
\end{defi}

%%%%
%Let us now discuss the operators $J^2_{\mb a,N}$ and $J^2_{\ovl{\mb a},N}$. 
The following derivation will be helpful in the proof of the Main Theorem in the next section. Let $R_{kl}(\zeta)$ be the (complex) rotation by $\zeta\in\C$ in the complex $kl$-plane. In particular, in the $a_1a_2$-plane, we have 
\begin{align*}
	R_{12}(\zeta)=\begin{bmatrix}
		\cos(\zeta)&-\sin(\zeta)&\cdots\\
		\sin(\zeta)&\cos(\zeta)&\cdots\\
		\vdots&\vdots&\ddots
	\end{bmatrix}
\end{align*}

We can consider $R_{kl}(\zeta)$ as an element of $\SO(N,\C)$ acting on $C^1(Q^{N-1}(b))$ by the rule
\begin{align*}
	R_{kl}(\zeta)\, f(\mb a,\ovl{\mb a})= f\left(R^{-1}_{kl}(\zeta)\mb a,\ovl{R^{-1}_{kl}(\zeta)\mb a}\right).
\end{align*}
Let $J_{kl}$ be the infinitesimal generator of $R_{kl}(\zeta)$. We have
\begin{align}
	(J_{kl}f)(\mb a,\ovl{\mb a})
	&=\left.\frac{\pt}{\pt \zeta}\right|_{\zeta=0}(R_{kl}(\zeta) f)(\mb a,\ovl{\mb a})
	%=\frac{1}{2}\left(\left.\frac{\pt}{\pt\theta}-i\frac{\pt}{\pt\ve}\right)\right|_{\theta=\ve=0}f\left(R^{-1}_{kl}(\zeta)\mb a,\ovl{R^{-1}_{kl}(\zeta)\mb a}\right)
	=\left(a_l\frac{\pt}{\pt a_k}-a_k\frac{\pt}{\pt a_l}\right)f(\mb a,\ovl{\mb a}).\label{eqn:rot_generator}
\end{align}
Now, we define the following operators on $C^2 (Q^{N-1}(b))$
\begin{align*}
	&J_{\mb a,N}^2=-\sum_{1\le k<l\le N} J_{kl}^2=-\sum_{1\le k<l\le N} \left(a_k\frac{\pt}{\pt a_l}-a_l\frac{\pt}{\pt a_k}\right)^2,\\
	&J_{\ovl{\mb a},N}^2=-\sum_{1\le k<l\le N} \ovl{J_{kl}}^2=-\sum_{1\le k<l\le N} \left(\bar a_k\frac{\pt}{\pt \bar a_l}-\bar a_l\frac{\pt}{\pt \bar a_k}\right)^2.
\end{align*}
The negative signs here are to be consistent with Hall and Mitchell \cite{HallMit2002}. %Physicists prefer this choice of sign because the differential operators $iJ_{kl}$ are self-adjoint with respect to the $L^2$-inner-product induced by the Riemmanian volume form of $Q^{N-1}(b)$, so $J_{\mb a,N}^2=\sum_{k<l} (iJ_{kl})^2$ will have a minus sign.

The following result shows the relation between $J_{\mb a,N}^2$ and the hyperbolic heat equation from \cite[Lem.~4]{HallMit2002} (after rescaling by a factor of $b$).
\begin{prop}\label{prop:Gamma-p-function}
	Let $f(\mb p)$ be a smooth, radial, real-valued function on $\R^N$, independent of $\mb x$, such that the function $\tilde f(p=|\mb p|):=f(\mb p)$ is a smooth even function in $p$. Then 
	\begin{align}\label{eqn:hyp-heat-eqn}
		\left(J_{\mb a,N}^2f\right)(2\mb p)=\left(J_{\ovl{\mb a},N}^2f\right)(2\mb p)=\left[b^2\frac{\pt^2}{\pt r^2} + b(N-2)\coth \left( \frac {r}{b}\right)\frac{\pt}{\pt r}\right]_{r=2p}\tilde f(r).
	\end{align}
	In particular, this is true for $f(\mb p)=\theta^{N-1}_b(s,p)$ for any $s>0$.
\end{prop}
(There is a typographical error of the sign in \cite[Lem.~4]{HallMit2002}, but the proof of the result has the correct sign as stated in Prop.~\ref{prop:Gamma-p-function}) 

Recall the hyperbolic volume measure $\be^{N-1}_b$ in Prop.~\ref{prop:hyperbol-inv-measure}. We have the following.
\begin{prop}[{\cite[Lem.~3]{HallMit2002}}]\label{prop:invariant-action-quadric}
	For any $b>0$, the measure 
	\begin{align*}
		dh^{N-1}_b(\mb x,\mb p):=2^{N-1}\be^{N-1}_b(2p)\,d\mb p \,d\ovl\sig^{N-1}_b(\mb x)
	\end{align*} 
	defines a normalized volume measure on the quadric $Q^{N-1}(b)$ that is invariant under the action of $\SO(N,\C)$.
\end{prop}

Finally, we obtain the following conclusion.
\begin{thm}[{\cite[Thm.~5]{HallMit2002}}]\label{thm:SB-sphere}
	For any $T>0$, denote the measure $\theta^{N-1}_b(2T,2p)\,dh^{N-1}_b(\mb x,\mb p)$ on the quadric $Q^{N-1}(b)$ by $d\nu^{N-1}_{b,T}$. Then, the Segal--Bargmann transform $\cal C^{N-1}_{b,T}$ in Def.~\ref{def:Segal-Bargmann sphere} is a unitary map from $L^2\big(S^{N-1}(b),\ovl\sig^{N-1}_b\big)$ onto the space $\cal HL^2\big(Q^{N-1}(b),\nu^{N-1}_{b,T}\big)$ of holomorphic functions $F$ on $Q^{N-1}(b)$ for which
	\begin{align*}
		\int_{\mb x\in S^{N-1}(b)}\int_{\mb x \cdot \mb p=0}|F(\mb a(\mb x,\mb p))|^2\,d\nu^{N-1}_{b,T}<\infty.
	\end{align*} 
\end{thm}

%========================================================================
\section{Large-$N$ Limits of Spherical Measures and Spherical Laplacians}\label{sec:large-N-sphere}
%========================================================================
%-------------------------------------------------
\subsection{Large-$N$ limit of spherical measures} \label{sec:poly-sphere-meas}
%-------------------------------------------------
In the sequel, when $b=\sqrt N$, we will simply write the normalized volume measure on $S^{N-1}(\sqrt N)$ as $\ovl\sig^{N-1}$ instead of $\ovl\sig^{N-1}_{\sqrt N}$. We now summarize the notion of convergence of the spherical measure $\ovl{\sig}^{N-1}$ to the standard Gaussian measure $\mu^\infty_1$. The relation between spherical measures on large sphere and Gaussian measure was studied in statistical mechanics (see, e.g., Boltzmann \cite{Boltzmann}). Later, there were several mathematical formulations done by Wiener \cite{Wie23}, L\'evy \cite{Levy22}, and Hida and Nomoto \cite{HiNo64}. The work by Umemura and Kono \cite{UmKo65} about limiting behaviors of spherical harmonics and the spherical Laplacian was probably the first attempt to study convergence of measures by letting the dimension of the sphere go to infinity. The approach by Petersen and Sengupta \cite{PeSen} is slightly different as the authors consider polynomials as functions in the Hilbert space $L^2\big(S^{N-1}(\sqrt{N}),\ovl\sig^{N-1}\big)$. 
%%%%

The main results of \cite{UmKo65} and \cite{PeSen} revolve around the following important observation, which also play a key role in understanding polynomials of complex variables.
\begin{lem}\label{lem:maps-of-sph-lap}
	For integers $k,N$ with $k<N$, and any polynomial $f\in \cal P_k(\R)$, there exists a unique polynomial $g\in\cal P_k(\R)$ such that 
	\begin{align*}
		\Del_{S^{N-1}(b)} f\big|_{S^{N-1}(b)}=g\big|_{S^{N-1}(b)}.
	\end{align*}
\end{lem}
In particular, if $b=\sqrt N$, then  Prop.~\ref{prop:spherical-laplacian} gives us 
\begin{align}
	g=\left(\Del_{\R^N}-(r\pt_r)_k+\frac{1}{N}(2(r\pt_r)_k-(r\pt_r)^2_k)\right)f.\label{eqn:sph-Lap-poly},
\end{align}
where $(r\pt_r)_k$ is as in Rem.~\ref{rem:polynomials}.

It is important that in the statement of the lem	ma we need the hypothesis $k<N$ and the conclusion that $g$ is the \emph{unique} polynomial of the same $k$ variables. This uniqueness in particular means that we can consider $\Del_{S^{N-1}(b)}$ as a map from $\cal P_k(\R)$ to itself, for all $k\in\N$.

\smallskip

For any $t>0$, recall the Gaussian measure $\mu^1_t$ on $\R$. Let $\cal B(\R^k)$ be the $\sig$-algebra of Borel sets on $\R^k$. For $k\le l$, we consider $\R^k$ as the set of the first $k$ coordinates in $\R^l$. The set $\cal A^0_t=\bigcup^{\infty}_{j=1}\cal B(\R^j)$ is an algebra. The theory of Infinite Product Measure (see, e.g., \cite[Sect.~3.5]{Bog2007}) shows that there exists a unique probability measure $\mu^\infty_t$ defined on $\sig( \cal A^0_t)$, the smallest $\sig$-algebra  generated by $\cal A^0_t$, such that 
\begin{align*}
	\mu^\infty_t(E) = \mu^k_t(E),\quad\textrm{for all } E\in\cal B(\R^k).
\end{align*}
The measure $\mu^\infty_t$ is compatible with $\mu^k_t$ in the sense that for any function $f\in L^1 \left(\R^k,\mu^k_t\right)$ considered as a function on $\R^\infty$, 
\begin{align*}
	\int_{\R^\infty} fd\mu^{\infty}_t=\int_{\R^k} fd\mu^k_t.
\end{align*}

Let us introduce notation
\begin{align*}
	M_t:=\bigcup_{k=1}^\infty L^2(\R^k,\mu^k_t).
\end{align*}
For $l\le k$, we can regard $L^2(\R^l,\mu^l_1)$ as a subspace of $L^2(\R^k,\mu^k_t)$ , so that 
\begin{align*}
	L^2(\R,\mu^1_t)\subset L^2(\R^2,\mu^2_t) \subset L^2(\R^3,\mu^3_t) \subset \dots \subset M_t\subset L^2(\R^\infty,\mu^\infty_t).
\end{align*}
The space $L^2(\R^\infty,\mu^\infty_t)$ satisfies the following property.
\begin{lem}\label{lem:dense}
	For any $t\in(0,\infty]$, the space $M_t$ is dense in $L^2(\R^\infty, \mu^\infty_t)$.
\end{lem}
\begin{proof}[Proof]
	Note that simple functions are dense in $L^2(\R^\infty,\mu^\infty_t)$, so it suffices to show that indicator functions of measurable sets are in the closure of $M_t$. Hence, let $\cal B$ be the collection of sets whose indicator functions are in $\ovl M_t$. By assumption, $\cal B$ contains the algebra $\cal A^0_t$. Moreover, by Monotone Convergence Theorem, $\cal B$ is a monotone class. Therefore, by Monotone Class Lemma, $\cal B$ is the whole $\sig$-algebra $\sig(\cal A^0_t)$.
\end{proof}
An equivalent statement of this result is that the space $L^2(\R^\infty, d\mu^\infty_t)$ is the closure with respect to $\mu^\infty_t$ of the set of all functions $f$ of \emph{finitely} many variables such that 
\begin{align*}
	\| f\|_{L^2(\R^\infty,\mu^\infty_t)}:=\left(\int_{\R^\infty} |f|^2\,d\mu^\infty_t\right)^{\frac{1}{2}}<\infty.
\end{align*}
We will use this property in the proof of the Main Theorem of the next section. 

Finally, we cite some results about the convergence of $\ovl \sig^{N-1}$ to $\mu^\infty_1$ from \cite[Prop.~1]{UmKo65} and \cite[Thm.~2.1]{PeSen}.

\begin{thm}\label{thm:limitsph}
Denote by $\ovl\sig^{N-1}_{k}$ $(N>k)$ the joint distribution of $x_1,x_2,\dots,x_k$, that is, for any Borel subset $E\subset \R^k$,
\begin{align*}
	\ovl\sig^{N-1}_k(E)=\ovl\sig^{N-1}\left(\{(x_1,\dots,x_k,\dots,x_N)\in S^{N-1}(\sqrt N)\colon (x_1,\dots,x_k)\in E\}\right).
\end{align*}
Then,
\begin{align*}
	\lim_{N\to\infty} \ovl\sig^{N-1}_k(E)=\mu^k_1(E).
\end{align*}
The convergence is uniform for all Borel subsets $E$. 

Furthermore, for $p_1,p_2\in\cal P_k(\R)$, we have 
\begin{align*}
		\lim_{N\to\infty} 
		\la p_1,p_2\ra_{L^2(S^{N-1}(\sqrt{N}),\ovl{\sig}^{N-1})}
		=\la p_1,p_2\ra_{L^2(\R^\infty,\mu^\infty_1)}.
	\end{align*}
\end{thm}

%----------------------------------------------------------------------------------------
\subsection{The Hermite differential operator and large-$N$ limit of spherical Laplacians}\label{sec:Hermite}
%----------------------------------------------------------------------------------------
Similar to the Euclidean case, we have the series expansion of the heat operator $e^{\frac{T}{2}\Del_{S^{N-1}(b)}}$.
\begin{prop}\label{prop:heat-powerseries-sphere}
	For any $T>0$ and $q\in\cal P_k(\R)$ considered as a function on $S^{N-1}(b)$, $N>k$, we have
	\begin{align*}
		e^{\frac{T}2\Del_{S^{N-1}(b)}} q=
		\sum_{n=0}^\infty \frac{1}{n!}\left(\frac{T}{2}\Del_{S^{N-1}(b)}\right)^n q,
	\end{align*}
	where the operator on the left-hand side is the heat operator defined in \eqref{eqn:heat-real-transform}.
\end{prop}	

Let us introduce the following notations:
\begin{align*}
	&\Del_{\R^\infty} := \sum_{j=1}^\infty \frac{\pt^2}{\pt x_j^2},\quad
	(r\pt_r)_\infty := \sum_{j=1}^\infty x_j\frac{\pt}{\pt x_j}.
\end{align*}
\begin{defi}\label{def:Hermite-op}
	We define the \emph{Hermite differential operator} on $\bigcup_{k=1}^\infty C^2(\R^k)$ to be
	\begin{align*}
		\mb H=\sum_{j=1}^\infty \left(\frac{\pt^2}{\pt x_j^2}-x_j\frac{\pt}{\pt x_j}\right)=\Del_{\R^\infty} -(r\pt_r)_\infty.
	\end{align*}
\end{defi}
\begin{prop}[{\cite[Prop.~3]{UmKo65},\cite[Prop.~5.4]{PeSen}}]\label{prop:Sen-Hermite}
	For any polynomial $q\in \cal P_k(\R)$, the function $\mb H\,q$ is in $\cal P_k(\R)$ and
	\begin{align*}
		\lim_{N\to\infty} \Del_{S^{N-1}(\sqrt N)} q=\mb  H\,q.
	\end{align*}
	The convergence is with respect to an arbitrary norm on the finite-dimensional vector space $\cal P_k(\R)$.
\end{prop}
A proof to this is simply by letting $N\to\infty$ in \eqref{eqn:sph-Lap-poly}. Another important of the Hermite differential operator $\mb H$ is that it is self-adjoint with respect to the inner product on $L^2(\R^\infty, d\mu^\infty_1)$ \cite[Prop.~5.1]{PeSen}. This suggests that $\mb H$ behaves like the ``Laplacian" on $\R^\infty$ with respect to $\mu^\infty_1$ (i.e., the Dirichlet form operator with respect to $\mu^\infty_1$).

%==============================
\section{Proof of Main Results}\label{chap:proof}
%==============================
%%%%
In this section, we will establish the proof of the Main Theorem (Thm.~\ref{thm:main}) following the commutative diagram at the end of Sect.~\ref{chap:summary}. 
In the sequel, let us fix a time $T>0$. We will establish results both for general $b$ and for the specific case $b=\sqrt N$. 
%For the latter case, we write the Segal--Bargmann transform on $S^{N-1}(\sqrt N)$ as $\cal C^{N-1}_T$ instead of $\cal C^{N-1}_{\sqrt N,T}$. Also, the volume measure of $S^{N-1}(\sqrt N)$ is denoted by $\ovl \sig^{N-1}$ instead of $\ovl \sig^{N-1}_{\sqrt N}$, and the quadric measure is denoted by $\nu^{N-1}_T$  instead of $\nu^{N-1}_{\sqrt N,T}$.

Several points have already been established in Petersen and Sengupta \cite{PeSen} and in Hall and Mitchell \cite{HallMit2002}: 
\begin{enumerate}[(i)]
	\item The upper horizontal arrow of the diagram is the content of  \cite{HallMit2002} and is stated in Thm.~\ref{thm:SB-sphere};
	\item The vertical arrow on the leftmost part of the diagram is the content of \cite{PeSen} and is stated in Thm.~\ref{thm:limitsph}. 
\end{enumerate}

We complete the proof by showing the following points in this section:
\begin{enumerate}[(i)]
	\item The lower horizontal arrow of the diagram is the content of Thm.~\ref{thm:main} (1), which is proved in Sect.~\ref{sec:limit-horizontal} (Thm.~\ref{thm:horizontal_arrow});
	\item The vertical arrow on the rightmost part of the diagram is the content of Thm.~\ref{thm:main} (2), which is proved in Sect.~\ref{sec:limit-quad-measure} (Thm.~\ref{thm:limitquad});
	\item The commutativity of the diagram is the content of Thm.~\ref{thm:main} (3), which is proved in Sect.~\ref{sec:limitmap} (Thm.~\ref{thm:limitmap}).
\end{enumerate}

%------------------------------------------------------------------------------------------------------------------------------
\subsection{The space $\cal H L^2(\C^\infty,\ga^\infty_T)$ and the Segal--Bargmann transform on $L^2(\R^\infty, \mu^{\infty}_1)$}\label{sec:limit-horizontal}
%------------------------------------------------------------------------------------------------------------------------------
Recall that $\cal H(\C^k)$ denotes the space of entire holomorphic functions on $\C^k$. As the complex counterpart of $(r\pt_r)_k$ for real functions, let us introduce the operator $(a\pt_a)_k$ for functions in $\cal H(\C^k)$, given by 
\begin{align}
	(a\pt_a)_k:=\sum_{j=1}^k a_j\frac{\pt}{\pt a_j}.\label{eqn:complex-Euler-op}
\end{align}
We also define the operator $(a\pt_a)_\infty$ as
\begin{align*}
(a\pt_a)_{\infty}:=\sum_{j=1}^\infty a_j\frac{\pt}{\pt a_j},
\end{align*}
which satisfies $(a\pt_a)_\infty\big|_{\cal H(\C^k)}=(a\pt_a)_k$. If $f\in C^\infty(\R^k)$ has a holomorphic extension $f_\C$, then
\begin{align}\label{eqn:commute}
	\left ( (r\pt_r)_k f\right )_\C = (a\pt_a)_k f_\C.
\end{align}

Let us now consider the exponentiation of $(a\pt_a)_k$. 
\begin{defi}\label{def:dilation}
	For any $\lam \in \C$, the operator $e^{\lam (a\pt_a)_k}\colon \cal H(\C^k)\to\cal H(\C^k)$ is given by 
	\begin{align*}
		(e^{\lam(a\pt_a)}f)(\mb a)=f(e^{\lam}\mb a),
	\end{align*}
	that is, $e^{\lam(a\pt_a)}$ acts on $f$ by dilating the variable $\mb a$ by a factor $e^{\lam}$.
	
	Similarly, for any $\lam\in\R$, the operator $e^{\lam (r\pt_r)_k}\colon C^\infty(\R^k)\to C^\infty(\R^k)$ is given by 
	\begin{align*}
		(e^{\lam(r\pt_r)_k}f)(\mb x)=f(e^{\lam}\mb x).
	\end{align*}
\end{defi}

Recall from Thm.~\ref{thm:heat-real-general} that, for $0<\frac{t}{2}<s<\infty$, the Segal--Bargmann transform  $\cal B_{s,t}^k$ defines a unitary map from $L^2(\R^N, \mu_1^k(\mb x))$ onto $\cal HL^2(\C^k,\xi^k_{s,t}(\mb u,\mb v))$, where
\begin{align*}
	\xi^{k}_{s,t}(\mb u+i\mb v)=(\pi(2s-t))^{-\frac{k}{2}}(\pi t)^{-\frac{k}{2}} e^{-\frac{\mb u^2}{2s-t}-\frac{\mb v^2}{t}}.
\end{align*}
Also, recall the measure $\ga^1_T$ from \eqref{eqn:meas-gamma}. This gives rise to the following Gaussian measure $\ga^k_T$ on $\C^k=\R^{2k}$:
\begin{align*}
	d\ga^k_T(\mb u,\mb v)= \ga^k_T(\mb u+i\mb v)\,d\mb u\,d\mb v= (\pi(e^T+1))^{-\frac k2}(\pi(e^T-1))^{-\frac k2} e^{-\frac{\mb u^2}{e^T+1}}
	e^{-\frac{\mb v^2}{e^T-1}}\,d\mb u\,d\mb v.
\end{align*}

\begin{prop}\label{prop:isometry-quad}
	For any $k\in\N^*$, $T>0$, the map $e^{-\frac{T}{2}(a\pt_a)_k}\cal B^{k}_{1,\ 1-e^{-T}}$ is a unitary map from $L^2(\R^k,\mu^k_1)$ onto $\cal HL^2(\C^k,\ga^k_T)$. In particular, for any $f\in L^2(\R^k,\mu^k_1)$ we have
	\begin{align*}
		\int_{\R^k} |f(\mb x)|^2\,d\mu^k_1(\mb x)=\int_{\C^k}\left|e^{-\frac{T}{2}(a\pt_a)_k}\left(e^{\frac{1-e^{-T}}2\Del_{\R^k}}f\right)_\C(\mb u+i\mb v)\right|^2\,d\ga^k_T(\mb u,\mb v).
	\end{align*}
\end{prop}

\begin{lem}
	For any holomorphic polynomial $q\in\cal H\cal P_k(\C)$ and $\lam\in\C$, we have 
	\begin{align*}
		e^{\lam (a\pt_a)_k}q=\sum_{n=0}^\infty \frac{\lam^n}{n!} (a\pt_a)^n_kq.
	\end{align*}
\end{lem}
\begin{proof}
	If $q_m\in\cal H\cal P_k(\C)$ is a holomorphic \emph{homogeneous} polynomial of degree $m$, then we can check that $(a\pt_a)_k q_m=mq_m$, so 
	\begin{align*}
		\sum_{n=0}^\infty \frac{\lam^n}{n!} (a\pt_a)^n_k\,q_m(\mb a)= e^{\lam m}\,q_m(\mb a)=q_m(e^{\lam}\mb a)=e^{\lam(a\pt_a)_k}\,q_m(\mb a).
	\end{align*} 
	Since any holomorphic polynomial is a sum of holomorphic homogeneous polynomials, the equality above is true for all $q\in\cal H\cal P_k(\C)$.
\end{proof}

We can also verify the following.
\begin{lem}\label{lem:real-SB-with-dilation}
	For any $\lam$ \underline{\emph{real}}, the map 
	\begin{align*}
		e^{\lam (a \pt_a)_k}\colon \cal H L^2(\C^k,\xi^k_{s,t})&\to \cal H L^2(\C^k,\xi^k_{s',t'})\\
		f(\mb a)&\mapsto f(e^\lam \mb a)
	\end{align*}
	is a unitary map, where $s'=se^{-2\lam}$ and $t'=te^{-2\lam}$.
\end{lem}

\begin{proof}[Proof of Prop.~\ref{prop:isometry-quad}]
	With $s=1$, $t=1-e^{-T}$, and $\lam=-\frac{T}{2}$ in Lem.~\ref{lem:real-SB-with-dilation}, noting that $0<t<2s$, we have 
	\begin{align*}
		\xi^k_{s',t'}(\mb u+i\mb v)
		%&= e^{2k\lam}(\pi (2s-t))^{-\frac k2}(\pi t)^{-\frac k2}e^{-e^{2\lam}\frac{\mb u^2}{2s-t}}e^{-e^{2\lam}\frac{\mb v^2}t}\\
		&= e^{-kT}(\pi(1+e^{-T}))^{-\frac k2}(\pi (1-e^{-T}))^{-\frac k2}e^{-\frac{e^{-T}\mb u^2}{1+e^{-T}}}e^{-\frac {e^{-T}\mb v^2}{1-e^{-T}}}
		%&= (\pi(e^T+1))^{-\frac k2}(\pi (e^T-1))^{-\frac k2}e^{-\frac{\mb u^2}{e^T+1}}e^{-\frac{\mb v^2}{e^T-1}}
		=\ga^k_T(\mb u,\mb v).
	\end{align*}		
	The map $e^{-\frac{T}{2}(a\pt_a)_k}\cal B^{k}_{1,\ 1-e^{-T}}$ is unitary because it is the composition of the dilation $e^{-\frac{T}{2}(a\pt_a)_k}$ (Lem.~\ref{lem:real-SB-with-dilation}) and the Segal--Bargmann transform $\cal B^{k}_{1,\ 1-e^{-T}}$, which are both unitary. 
\end{proof}
Note that the proof also shows that $\ga^k_T$ is a special case of $\xi^k_{s,t}$.

Similar to the measure $\mu^\infty_t$, we can construct the probability measure $\xi^\infty_{s,t}$ which is the infinite product of the measure $\xi^1_{s,t}$ on $\C$. The measure of $\xi^\infty_{s,t}$ is compatible with $\xi^k_{s,t}$ for any $k\in\N^*$, in the sense that for any function $f$ of $2k$ complex variables $a_1,\dots,a_k,\ovl a_1,\dots,\ovl a_k$ in $L^1(\C^k,\xi^k_{s,t})$ considered also as a function on $\C^\infty$, then 
\begin{align*}
	\int_{\C^\infty} f\,d\xi^\infty_{s,t}=  \int_{\C^k} f\,d\xi^k_{s,t}.
\end{align*}
Moreover, with the same construction of the space $L^2(\R^\infty,\mu^\infty_t)$ in Sect.~\ref{sec:poly-sphere-meas}, we can show that there exists an inner product space $L^2(\C^\infty,\xi^\infty_{s,t})$ that extends $L^2(\C^k,\xi^k_{s,t})$ for all $k\in\N^*$, i.e, for any $F,G$ in $L^2(\C^k,\xi^k_{s,t})$,
\begin{align}\label{eqn:inn-prod-gam}
	\la F,G\ra_{L^2(\C^\infty,\xi^\infty_{s,t})}=\la F,G\ra_{L^2(\C^k,\xi^k_{s,t})}:=\int_{\C^k} F\ovl G\,d\xi^k_{s,t}.
\end{align} 
The same argument works if we replace $\xi^\infty_{s,t}$ by $\ga^\infty_T$.

\begin{defi}\label{def:HL-infinity}
	Define $\cal P_\infty(\R)$ and $\cal H\cal P_\infty(\C)$ as follows:
	\begin{align*}
		\cal P_\infty(\R):=\bigcup_{k=1}^\infty \cal P_k(\R),\quad \cal H\cal P_\infty(\C):=\bigcup_{k=1}^\infty \cal H\cal P_k(\C).
	\end{align*} 
	For any $s,t$ with $0<\frac{t}{2}<s<\infty$, we set the space $\cal H L^2(\C^\infty,\xi^\infty_{s,t})$ to be the \emph{closure} of $\cal H\cal P_\infty(\C)$ in the Hilbert space $L^2(\C^\infty,\xi^\infty_{s,t})$.
\end{defi}

\begin{rem} \label{rem:holomorphic}
	It is clear that the space $\cal HL^2(\C^\infty,\xi^\infty_{s,t})$ is a Hilbert subspace of $L^2(\C^\infty,\xi^\infty_{s,t})$. It may be more reasonable to define $\cal H L^2(\C^\infty,\xi^\infty_{s,t})$ as the closure of $\bigcup_{k=1}^\infty\cal H L^2(\C^k,\xi^k_{s,t})$ instead of the closure of $\bigcup_{k=1}^\infty\cal H\cal P_k(\C)$, but since holomorphic polynomials are dense in $\cal HL^2(\C^k,\xi^k_{s,t})$ for each $k$ (see, e.g. \cite[Thm 3.6]{DH99}), we actually do not obtain a new space.
\end{rem}
\begin{rem}
	Note that although the notation $\cal HL^2(\C^\infty,\xi^{\infty}_{s,t})$ has the prefix $\cal H$ which stands for ``holomorphic", a function $f$ in $\cal H L^2(\C^\infty,\ga^\infty_T)$ (of a complex sequence $\mb a\in \C^\infty$) may fail to be holomorphic in any practical sense. For example, with $\xi^k_{s,t}=\ga^k_T$, let us consider the series $f(\mb a)=\sum_{j=1}^\infty c_ja_j$ where $\sum_{j=1}^\infty |c_j|^2<1$ and $c_j\ne 0$ for all $j$. We see that
	\begin{align*}
		\int_{\C^2} a_k \bar a_l\ d\ga^2_T(a_k,a_l,\bar a_k,\bar a_l)=2e^T\del_{jk},
	\end{align*}
	which gives $\|f\|_{L^2(\C^\infty,\ga^\infty_T)}=2e^T\sum_{j=1}^\infty|c_j|^2<\infty$, so $f$ is an element of $L^2(\C^\infty,\ga^\infty_T)$. Also, $f$ is a limit point of the sequence $\{f_n=\sum_{j=1}^n c_ja_j\}_{n}$ in $\cal H\cal P_{\infty}(\C)$, so it is an element of $\cal HL^2(\C^\infty,\ga_T^\infty)$ by completeness. However, $f$ is not even defined at $(c_1^{-1},c_2^{-1},\dots,c_n^{-1},\dots)$.
\end{rem}

\begin{prop}\label{prop:infty_operator}
	For any $\lam\in\R$, there exist unique unitary maps $\cal B^\infty_{s,t}$ from $L^2(\R^\infty,\mu^\infty_s)$ onto $\cal HL^2(\C^\infty,\xi^\infty_{s,t})$ and $e^{\lam(a\pt_a)_\infty}$ from $\cal HL^2(\C^\infty,\xi^\infty_{s,t})$ onto $\cal HL^2(\C^\infty,\xi^\infty_{se^{-2\lam},te^{2-\lam}})$ such that for any $k\in \N^*$,
	\begin{align*}
		&\cal B^\infty_{s,t}\, f = \cal B_{s,t}^k\, f,\qquad \qquad\forall f\in L^2(\R^k,\mu^k_s),\\
		&e^{\lam(a\pt_a)_\infty}\, g = e^{\lam (a\pt_a)_k}\, g, \qquad \forall g\in \cal HL^2(\C^k,\xi^k_{s,t}).
	\end{align*}
\end{prop}
\begin{proof}	
	Let us prove the existence of $\cal B^\infty_{s,t}$. Note that $\cal HL^2(\C^\infty,\ga^\infty_T)$ contains $\cal HL^2(\C^k,\ga^k_T)$ for all $k$ by Rem.~\ref{rem:holomorphic}. Hence, we can define a linear map $\cal L$ from $M_s$ to $\cal HL^2(\C^k,\xi^\infty_{s,t})$ by the rule $\cal L\,f = \cal B_{s,t}^k\,f$ for any $f\in L^2(\R^k,\mu^k_s)$. Since $M_s$ is dense in $L^2(\R^\infty, d\mu^\infty_s)$ (Lem.~\ref{lem:dense}), by the Bounded Linear Transformation (BLT) Theorem, there exists a unique bounded linear map $\cal L$ from the whole space $L^2(\R^\infty,\mu^\infty_s)$ to $\cal HL^2(\C^\infty,\xi^\infty_{s,t})$ such that $\cal B^\infty_{s,t}\big|_{M_s}=\cal L$. 
	
	Since the restriction of $\cal B^\infty_{s,t}$ to each $L^2(\R^k,\mu^k_s)$ is an isometry, it is an isometry on $M_s$. Consequently, the extended map $\cal B^\infty_{s,t}$ is an isometry from $\overline{M_s}=L^2(\R^\infty,\mu^\infty_s)$ (Lem.~\ref{lem:dense}) to $\cal HL^2(\C^\infty,\xi^\infty_{s,t})$. Since $L^2(\R^\infty,\mu^\infty_s)$ is a Hilbert space, the graph of $\cal B^\infty_{s,t}$  is closed and the image contains all $\cal HL^2(\C^k,\xi^k_{s,t})$ for all $k$, so map is surjective.
	
	The construction of the map $\cal T$ is similar with the sequence $\{\cal HL^2(\C^k,\xi^k_{s,t})\}_k$ replaced by $\{\cal HL^2(\C^k,\xi^k_{se^{-2\lam},te^{-2\lam}})\}_k$ and the sequence $\{ L^2(\R^k,\mu^k_s)\}_k$ replaced by $\{\cal HL^2(\C^k,\xi^k_{s,t})\}_k$.
\end{proof}	

We now have enough information to restate Point (1) of the Main Theorem \ref{thm:main}. Recall the Hermite operator $\mb H$ from Def.~\ref{def:Hermite-op}, we want to show the following theorem for the rest of the subsection.
\begin{thm}\label{thm:horizontal_arrow}
	For any $T>0$ and $k\in\N^*$, the following power series is convergent for all $p\in\cal P_k(\R)$:
	\begin{align}\label{eqn:power-series}
		e^{\frac{T}{2}\mb H}\,p:=\sum_{n=0}^{\infty}\frac{1}{n!}\left(\frac{T}{2}\mb H\right)^n\,p,
	\end{align}
	Define the linear map $\mb B_T\colon \cal P_\infty(\R) \to L^2(\C^\infty,\xi^\infty_{s,t})$ by
	\begin{align*}
		\mb B_T\,p:=\left(e^{\frac{T}{2}}\,p\right)_\C,\qquad p\in\cal P_\infty(\R).
	\end{align*}
	then $\mb B_T$ can be extended uniquely to a unitary map from $L^2(\R^\infty,\mu^\infty_1)$ onto $\cal HL^2(\C^\infty,\ga^\infty_T)$, and the extension agrees with the operator $e^{-\frac{T}{2}(a\pt_a)_\infty}\cal B^{\infty}_{1,\,1-e^{-T}}$.
\end{thm}

%%%%%
\begin{lem}[{\cite[Chap. 5, Ex. 5]{Hall2015}}]\label{lem:exercise}
	Let $X,Y$ be two linear operators acting on a finite dimensional vector space. Suppose $[X,Y]=\al Y$ where $\al$ is not an integer multiple of $2\pi i$, then
	\begin{align*}
		e^{X}e^{Y}&=e^{X+\frac{\al}{1-e^{-\al}}Y},\\
		e^{Y}e^{X}&=e^{X-\frac{\al}{1-e^{\al}}Y},\textrm{ and}\\
		e^{X+Y}&=e^{X+\frac{1-e^{-\al}}{\al}Y}.
	\end{align*}
\end{lem}
\begin{lem} \label{lem:commutative-relation}
	For any polynomial $q\in\cal P_k(\R)$, we have
	$[(r\pt_r)_k,\Del_{\R^k}]q=-2\Del_{\R^k}q$.
\end{lem}
\begin{proof}
	We need only to verify this for homogeneous polynomials. If $q_m$ is a homogeneous polynomial of degree $m$, then $(r \pt_r)_k\,q_m= mq_m$, while the degree of $\Del_{\R^k} q_m$ is two less than that of $q_m$, so that 
	\begin{align*}
		(r\pt_r)_k\, \Del_{\R^k}\, q_m-\Del_{\R^k}\, (r\pt_r)_k\, q_m=(m-2)\Del_{\R^k}\, q_m-\Del_{\R^k}\, mq_m=-2\Del_{\R^k}\, q_m.
	\end{align*}
	Hence, $[(r \pt_r)_k,\Del_{\R^k}]=-2\Del_{\R^k}$.
\end{proof}

\begin{proof}[Proof of Thm.~\ref{thm:horizontal_arrow}]
	Let $l$ be the degree of $p$. We can consider $\Del_{\R^k}$, $(r\pt_r)_k$, and $\mb H$ as bounded linear operator on the finite-dimensional vector space $\cal P^{\le l}_k(\R)$, so $\mb H$ acting on $\cal P^{\le l}_k(\R)$ has a matrix representation. Hence, for any $T>0$, the power series defining $e^{\frac{T}{2}\mb H}$ converges. In particular, $e^{\frac{T}{2}\mb H}p$ is convergent and is an element of $\mathcal P^{\le l}_k(\R)$.
	
	To compute $e^{\frac{T}{2}\mb H}$, still considering the action of bounded operators on the vector space $\cal P^{\le l}_k(\R)$, we use the third identity of Lem.~\ref{lem:exercise} with $X=-\frac{T}{2}(r\pt_r)_k$, $Y=\frac{T}{2}\Del_{\R^k}$, $\al=T$ (which comes from Lem.~\ref{lem:commutative-relation}) to obtain
	\begin{align*}
		e^{\frac{T}{2}(\Del_{\R^k}-(r\pt_r)_k)}\,p= e^{-\frac{T}{2}(r\pt_r)_k}e^{\frac{1-e^{-T}}{2}\Del_{\R^k}}\,p.
	\end{align*} 
	Note that by Thm.~\ref{thm:heat-powerseries}, $e^{\frac{1-e^{-T}}{2}\Del_{\R^k}}\,p$ is a polynomial so it can be holomorphically extended to the whole $\C^k$, and the remark before Def.~\ref{def:dilation} implies that $(e^{\lam (r\pt_r)_k}f)_\C=e^{\lam (a\pt_a)_k}f_\C$ if $f$ has an entire holomorphic extension. Therefore,
	\begin{align*}
		\mb B_T\,p=\left(e^{\frac{T}{2}\mb H}\,p\right)_\C&=\left( e^{-\frac{T}{2}(r\pt_r)_k}e^{\frac{1-e^{-T}}{2}}\,p\right)_\C 
		= e^{-\frac{T}{2}(a\pt_a)_k} \left(e^{\frac{1-e^{-T}}{2}\Del_{\R^k}}\,p\right)_\C
		= e^{-\frac{T}{2}(a\pt_a)_k} \cal B^{k}_{1,\,1-e^{-T}}\,p.
	\end{align*}
	The last expression in the equality above is known to be independent of the degree of $p$, so the first statement of the theorem is true for all $p\in\cal P_k(\C)$.
	
	Finally, since the set of polynomial is dense in $L^2(\R^\infty,\mu^\infty_1)$, the uniqueness of extension in Bounded Linear Transformation Theorem implies that $\mb B_T$ and $e^{-\frac{T}{2}(a\pt_a)_\infty}\cal B^{\infty}_{1,\,1-e^{-T}}$ agree on the whole $L^2(\R^\infty,\mu^\infty_1)$.
	%The computation above also shows that $\mb B_T$ and $e^{-\frac{T}{2}(a\pt_a)_\infty}\cal B^{\infty}_{1,\,1-e^{-T}}$ agree on $\cal P_k(\R)$ for each $k$ and is a unitary map from $\cal P_k(\R)$ to $\cal HL^2(\C^\infty,\ga^\infty_T)$. Since $\cal P_k(\R)$ is dense in $L^2(\R^k,\mu^{k}_1)$ with respect to the measure $\mu^{\infty}_1$, the extension is unique on $L^2(\R^\infty,\mu^\infty_1)=\ovl{M_1}=\ovl{\bigcup_{k=1}^\infty \cal P_k(\R)}$ by the uniqueness of Bounded Linear Transformation (BLT) Theorem.
\end{proof}

%----------------------------------------------
\subsection{Large-$N$ limit of quadric measures}\label{sec:limit-quad-measure}
%-----------------------------------------------
In this subsection, whenever we fix $b=\sqrt N$, we will use the notation $\nu^{N-1}_T$ in place of $\nu^{N-1}_{\sqrt N,T}$. Recall the definition of $Q^{N-1}(b)$ in \eqref{eqn:quadric-def}. The goal of this subsection is to prove the following theorem.
\begin{thm}\label{thm:limitquad}
	Suppose $q\in\cal P_k(\C)$, viewed also as a function on $\C^N$ for $N>k$. Then 
	\begin{align*}
		\lim_{N\to\infty} \int_{Q^{N-1}(\sqrt N)} q\,d\nu^{N-1}_{T}= \int_{\C^\infty} q\,d\ga^\infty_T.
	\end{align*}
	In particular, for $q_1,q_2\in \cal P_k(\C)$, we have
	\begin{align*}
		\lim_{N\to\infty} 
		\la q_1,q_2\ra_{L^2(Q^{N-1}(\sqrt{N}),\nu^{N-1}_{T})}
		=\la q_1,q_2\ra_{L^2(\C^\infty,\ga^\infty_k)}.
	\end{align*}
\end{thm}
Note that in Thm.~\ref{thm:limitquad}, the result does not involve any assumption on the holomorphicity of the polynomials in $\cal P_k(\C)$.

Let us establish several results for general $b>0$ before specifying $b=\sqrt N$. Recall from Sect.~\ref{sec:SB-sphere} the measure $dh^{N-1}_b(\mb x,\mb p):=2^{N-1}\be^{N-1}_b(2p)\,d\mb p\,d\ovl{\sig}^{N-1}_b(\mb x)$ on the quadric $Q^{N-1}$(b). 

\begin{defi} We define the differential operator 
\begin{align}
	\Ga_N:=\frac{1}{2}(J^2_{\mb a,N}+J^2_{\ovl{\mb a},N})\label{eqn:Ga_N}
\end{align} 
where $J^2_{\mb a,N}$ and $J^2_{\ovl{\mb a},N}$ are given in Prop.~\ref{prop:Gamma-p-function}.
\end{defi}

We have the following first observation about $\Ga_N$.
\begin{prop}[Integration by parts with $\Ga_N$] \label{prop:symmetric-op}
	If $f\in C^2 (Q^{N-1}(b))$ and $g$ is smooth and compactly supported on $Q^{N-1}(b)$, then 
	\begin{align*}
		\int_{Q^{N-1}(b)}\left(\Ga_N f\right)\,g\, dh^{N-1}_b
		=\int_{Q^{N-1}(b)} f\,\left(\Ga_N g\right) \, dh^{N-1}_b 
	\end{align*}
\end{prop}
\begin{proof}
	We use the notion of complex rotation $R_{kl}(\zeta)$ in \eqref{eqn:rot_generator}. The measure $dh^{N-1}_b$ is invariant under the action of $\SO(N,\C)$ by Prop.~\ref{prop:invariant-action-quadric}, so 
	\begin{align*}
		&\int_{Q^{N-1}(b)}f(R^{-1}_{kl}(\zeta)\mb a,\ovl{R^{-1}_{kl}(\zeta)\mb a}) \,g(R^{-1}_{kl}(\zeta)\mb a,\ovl{R^{-1}_{kl}(\zeta)\mb a}) \, dh^{N-1}_b(\mb a,\ovl{\mb a})\\ 
		&=\int_{Q^{N-1}(b)}f(\mb a,\ovl{\mb a}) \,g(\mb a,\ovl{\mb a}) \, dh^{N-1}_b(R_{kl}(\zeta)\mb a,\ovl{R_{kl}(\zeta)\mb a})
		=\int_{Q^{N-1}(b)}f(\mb a,\ovl{\mb a}) \,g(\mb a,\ovl{\mb a}) \, dh^{N-1}_b(\mb a,\ovl{\mb a}) <\infty.
	\end{align*} 
	Taking partial derivative $\pt/\pt\zeta$ under the integral sign and letting $\zeta$ go to $0$, we have 
	\begin{align*}
		&\int_{Q^{N-1}(b)}\left(J_{kl} f\right)\,g\, dh^{N-1}_b
		+\int_{Q^{N-1}(b)} f\,\left(J_{kl} g\right) \, dh^{N-1}_b=0,
		%\implies &\int_{Q^{N-1}(b)}\left(J_{kl} f\right)\,g\, dh^{N-1}_b
		%=-\int_{Q^{N-1}(b)} f\,\left(J_{kl} g\right) \, dh^{N-1}_b=0.
	\end{align*}
	where $J_{kl}=a_l\frac{\pt}{\pt a_k}-a_k\frac{\pt}{\pt a_l}$. Since $J^2_{\mb a,N}=-\sum_{k<l} J_{kl}^2$, we have 
	\begin{align}
		\int_{Q^{N-1}(b)}\left(J^2_{\mb a,N} f\right)\,g\, dh^{N-1}_b
		=\int_{Q^{N-1}(b)} f\,\left(J^2_{\mb a,N} g\right) \, dh^{N-1}_b.\label{eqn:symmetric-1}
	\end{align}
	Applying \eqref{eqn:symmetric-1} for $\bar f$ and $\bar g$ then taking the complex conjugation, we have 
	\begin{align}
		\int_{Q^{N-1}(b)}\left(J^2_{\ovl{\mb a},N} f\right)\,g\, dh^{N-1}_b
		=\int_{Q^{N-1}(b)} f\,\left(J^2_{\ovl{\mb a},N} g\right) \, dh^{N-1}_b.\label{eqn:symmetric-2}
	\end{align}
	Adding \eqref{eqn:symmetric-1} and \eqref{eqn:symmetric-2} gives us the conclusion.
\end{proof}

The next observation is that we can extend the integration by parts formula to the heat kernel $g=\theta^{N-1}_{b}$, even though it is not compactly supported on $Q^{N-1}(b)$. Recall that $\cal P_k(\C)$ denotes the set of (complex) polynomials depending on $2k$ variables $(\mb a,\ovl {\mb a})=(a_1,\dots,a_k,\bar a_1,\dots,\bar a_k)$.
\begin{prop}\label{prop:int-by-part-hyperbolic}
	For any polynomial $q(\mb a,\ovl{\mb a})\in \cal P_k(\C)$, we associate the function 
	\begin{align*}
		f_q(\mb x,\mb p):=q(\mb a(\mb x,\mb p),\ovl{\mb a}(\mb x,\mb p)).
	\end{align*} 
	Then,
	\begin{align*}
		\int_{\mb x\in S^{N-1}(b)}&\int_{T_{\mb x}S^{N-1}(b)}(\Ga_N f_q)(\mb x,\mb p)\,\theta^{N-1}_{b}(2T,2p)\,dh^{N-1}_b(\mb x, \mb p)\\
		&=\int_{\mb x\in S^{N-1}(b)}\int_{T_{\mb x}S^{N-1}(b)} f_q(\mb x,\mb p)\,
		\left(\Ga_N \theta^{N-1}_{b}\right)(2T,2p)\,dh^{N-1}_b(\mb x, \mb p).
	\end{align*} 
\end{prop}

%%%%%%%%%%%% The following estimate is now unnecessary
\begin{comment}
	\begin{lem}[Estimates of the gradient of the heat kernel {\cite{Kots07}}]\label{lem:Hamilton-complete}
		Suppose $(M^n,g)$ is a complete manifold with Ricci curvature $\tnm{Ric}\ge -K\cdot g$ for some $K\ge 0$. If $0<u(x,t)\le C$ is a solution to the heat equation on $M^n\times [0,t_0]$ for $0<t_0\le \infty$, i.e.
		\begin{align*}
			\frac{\pt}{\pt t} u(x,t)=\frac{1}{2}\Del_{M} u(x,t).
		\end{align*}
		then 
		\begin{align*}
			t|\nabla_M u|^2\le |u|^2 (1+Kt)\log\left(\frac Cu\right).
		\end{align*} 
	\end{lem}
\end{comment}
%%%%%%%%%%%%%

\begin{lem}[Recurrence relation of the hyperbolic heat kernel {\cite[Sect.~5.7.]{Davies}}]\label{lem:recur-relation}
	For any $d\ge 1$, $r\ge0$, $t>0$, we have 
	\begin{align*}
		\theta^{d+2}_{1}(s,r) = -\frac{e^{-\frac{ds}{2}}}{2\pi \sinh r}\frac{\pt}{\pt r} \theta^d_1 (s,r).
	\end{align*}	
\end{lem}

\begin{proof}[Proof of Prop.~\ref{prop:int-by-part-hyperbolic}]
	There exists a smooth function $\psi_0\colon \R\to \R$ such that
	\begin{itemize}
		\item $\psi_0(y)=1$ if $y\le 0$, 
		\item $\psi_0(y)=0$ if $y\ge 1$,
		\item $\psi_0$ is decreasing on $[0,1]$.
	\end{itemize}
	(See, e.g., \cite[Sect.~13.1]{Tu}) For any $R>0$ define $\psi_R\colon \R\to \R$ by the translation relation \begin{align*}
		\psi_R(y)=\psi_{0}(y-R).
	\end{align*}
	From the smoothness of $\psi_0$ and the translation relation, we see that there is a constant $L>0$ independent of $R$ such that $\big|\frac{\pt}{\pt p}\psi_R(y)\big|\le L$ and $\big|\frac{\pt^2}{\pt p^2}\psi_R(y)\big|\le L$.
	
	Now, for any $\mb x\in S^{N-1}(b)$, we define the function $\vp_{\mb x,R}\colon T_{\mb x} S^{N-1}(b)\to \R$ by 
	\begin{align*}
		\vp_{\mb x,R}(\mb p)=\psi_R(|\mb p|),
	\end{align*}
	then $\vp_{\mb x,R}$ is smooth, independent of $\mb x$, rotation-invariant in $\mb p$, and is compactly supported with $\vp_{\mb x,R}(\mb p)=0$ when $|\mb p|>R+1$.
	
	By Prop.~\ref{prop:symmetric-op}, we have 
	\begin{align*}
		\int_{\mb x\in S^{N-1}(b)}&\int_{\mb x \cdot \mb p=0}(\Ga_N f_q)(\mb x,\mb p)\,\theta^{N-1}_{b} (2T,2p)\,dh^{N-1}_b(\mb x, \mb p)\\
		&=\lim_{R\to\infty} \int_{\mb x\in S^{N-1}(b)}\int_{\mb x \cdot \mb p=0}(\Ga_N f_q)(\mb x,\mb p)\,\vp_{\mb x,R}(2p)\,\theta^{N-1}_{b} (2T,2p)\,dh^{N-1}_b(\mb x, \mb p)\\
		&=\lim_{R\to\infty} \int_{\mb x\in S^{N-1}(b)}\int_{\mb x \cdot \mb p=0} f_q(\mb x,\mb p)\,\Ga_N \left(\vp_{\mb x,R}(2p)\,\theta^{N-1}_{b} (2T,2p)\right)\,dh^{N-1}_b(\mb x, \mb p).
	\end{align*}
	By Prop.~\ref{prop:Gamma-p-function} and the product rule, we have 
	\begin{align}
		&\lim_{R\to\infty} \int_{\mb x\in S^{N-1}(b)}\int_{\mb x \cdot \mb p=0} f_q(\mb x,\mb p)\,\Ga_N \left(\vp_{\mb x,R}(2p)\,\theta^{N-1}_{b} (2T,2p)\right)\,dh^{N-1}_b(\mb x, \mb p)\label{eqn:est-integrals}\\
		&=\lim_{R\to\infty} \int_{\mb x\in S^{N-1}(b)}\int_{\mb x \cdot \mb p=0} f_q(\mb x,\mb p)\,\vp_{\mb x,R}(2p)\, \left(\Ga_N \theta^{N-1}_b\right)(2T,2p)\,dh^{N-1}_b(\mb x, \mb p)\notag\\
		&\quad +\lim_{R\to\infty} \int_{\mb x\in S^{N-1}(b)}\int_{\mb x \cdot \mb p=0} f_q(\mb x,\mb p)\, 
		\theta^{N-1}_{b}(2T,2p)\,(\Ga_N\vp_{\mb x,R})(2p)\,dh^{N-1}_b(\mb x, \mb p)\notag\\
		%\left[b^2\frac{\pt^2}{\pt r^2}+\right.\\
		%&\hspace{10em} \left.+a(N-2)\coth \left( \frac{r}{b}\right) \frac{\pt}{\pt r}\right]_{r=2p}\vp_{\mb x,R}(r)\,dh^{N-1}_b(\mb x, \mb p)	\\
		&\quad +2b^2\lim_{R\to\infty} \int_{\mb x\in S^{N-1}(b)}\int_{\mb x \cdot \mb p=0} f_q(\mb x,\mb p)\, \left.\frac{\pt}{\pt r}\right|_{r=2p}\, \left.\theta^{N-1}_{b}(2T,r)\frac{\pt}{\pt r}\right|_{r=2p}\vp_{\mb x,R}(r)\,dh^{N-1}_b(\mb x, \mb p).\notag
	\end{align}
	The last term comes from the second-order derivative term in \eqref{eqn:hyp-heat-eqn}. 
	
	Let us denote by $I_2$ and $I_3$ the second integral and the third integral on the right-hand side of \eqref{eqn:est-integrals}, respectively. First, let us estimate $I_2$. 
	Recall from Prop.~\ref{prop:quadric-diffeomorphic-cotang} that 
	\begin{align*}
		\mb a=\cosh \left(\frac {p}{b}\right) \mb x + i\frac {b}{p} \sinh \left(\frac{p}{b}\right) \mb p.
	\end{align*}	
	Since $|\mb x|=b$, and since $\cosh(p/b)$ and $\sinh(p/b)$ both grow exponentially in $p$, we see that $f_q(\mb x,\mb p)$ as a finite sum of hyperbolic sines and hyperbolic cosines of $p/b$ grows at most exponentially, i.e.,
	\begin{align}
		|f_q(\mb x,\mb p)|	\le C_1e^{C_2p} \label{eqn:fq-estimate},					
	\end{align}
	for some constants $C_1,C_2>0$ depending on $\deg(q),N,b$. Recall that the derivatives $\frac{\pt}{\pt r}\vp_{\mb x,R}$ and $\frac{\pt^2}{\pt r^2}\vp_{\mb x,R}$ are bounded by $L$ for all $R\ge 0$. Also, for $y\ge 1$, $\coth (y)\le 2$. Hence, using Prop.~\ref{prop:Gamma-p-function} again, when $\frac{R}{2}<p<\frac{R+1}{2}$ and $R$ is large,
	\begin{align}
		\left|\Ga_N\vp_{\mb x,R}\right| 
		&\le b^2\left|\frac{\pt^2}{\pt r^2}\vp_{\mb x,R}(r)\right|_{r=2p} + b(N-2)\left|\coth \left( \frac {r}{b}\right)\frac{\pt}{\pt r}\vp_{\mb x,R}(r)\right|_{r=2p}
		\le  C_{3},\label{eqn:Ga-phi-estimate}
	\end{align}
	for some constant $C_{3}>0$ depending on $N,b,L$. In addition, the function $\be_b^{N-1}$ defining $dh^{N-1}_b$ (Prop.~\ref{prop:hyperbol-inv-measure}) satisfies if $p>1$, then
	\begin{align}
		2^{N-1}\be^{N-1}_{b}(2p) \le 2e^{\frac{2(N-2)p}{b}}\label{eqn:beta-estimate}.
	\end{align}

	Note that $\Ga_N\vp_{\mb x,R}(p)=0$ on $\R\setminus(\frac{R}{2},\frac{R+1}{2})$,  for $R$ large enough,
	\begin{align}
		&|I_2|\le NC_1C_3\, \int_{\mb x\in S^{N-1}(b)} \int_{\frac{R}{2}<p<\frac{R+1}{2}} e^{C_2p}\,
		\theta^{N-1}_{b}(2T,2p)\,2e^{\frac{2(N-2)p}{b}}\, d\ovl{\sig}^{N-1}_b(\mb x)\notag\\
		& \le \int_{\frac{R}{2}}^{\frac{R+1}{2}}e^{C_{4}p}\theta^{N-1}_{b}(2T,2p)\,dp\label{eqn:upper-est-I2}.
	\end{align}
	for some constant $C_{4}>\max\{C_{2},2(N-2)/b\}$. With the upper-bound from  Prop.~\ref{prop:estimate-heat-hyperbolic}, 	Eq.~\eqref{eqn:upper-est-I2} gives us
	\begin{align}
		|I_2| \le k(N-1,2T)\int_{\frac{R}{2}} ^{\frac{R+1}{2}}\, e^{C_{4}p}\, e^{-\frac{p^2}{T}}\,dp.\label{eqn:upper-est-I2-2}
	\end{align}
	Since we are considering $T$ fixed, the term $k(N-1,2T)$ is also a constant, so \eqref{eqn:upper-est-I2-2} implies 
	\begin{align*}
		\lim_{R\to\infty} |I_2| = \lim_{R\to\infty} e^{C_4R}e^{-\frac{R^2}{T}}=0.
	\end{align*}
	
	\smallskip
	
	For the integral $I_3$, Lem.~\ref{lem:recur-relation} implies
	\begin{align*}
		%\coth\left(\frac{r}{b}\right)\frac{\pt}{\pt r} \theta^{N-1}_b(2T,r) = -\pi e^{\frac{(N-1)T}{b^2}}\sinh\left(\frac{2r}{a}\right) \theta^{N+1}_b(2T,r). 
		\frac{\pt}{\pt r} \theta^{N-1}_b(2T,r) = -2\pi e^{\frac{(N-1)T}{b^2}}\sinh\left(\frac{r}{b}\right) \theta^{N+1}_b(2T,r). 
	\end{align*}
	Now, $\sinh\left(\frac{r}{b}\right)\le e^{\frac{r}{b}}$ for $r$ large, and Prop.~\ref{prop:estimate-heat-hyperbolic} implies 
	\begin{align*}
		\left|\theta^{N+1}_b(2T,2p)\right|\le k(N+1,2T)\,e^{-\frac{p^2}{T}}
	\end{align*}
	for some constant $k(N+1,2T)$. Hence,
	\begin{align*}
		\left|\frac{\pt}{\pt r}\theta^{N-1}_b(2T,r)\right|_{r=2p}\le C_5\,e^{C_6p}\,e^{-\frac{p^2}{T}},
	\end{align*}
	for some constants $C_5,C_6>0$ depending on $b,N,T$, so using the same argument as the integral $I_2$, we have 
	\begin{align*}
		\lim_{R\to\infty}|I_3|=0.
	\end{align*}
	In conclusion,
	\begin{align*}
		&\int_{\mb x\in S^{N-1}(b)}\int_{\mb x \cdot \mb p=0}(\Ga_N f_q)(\mb x,\mb p)\,\theta^{N-1}_{b} (2T,2p)\,dh^{N-1}_b(\mb x, \mb p)\\
		&\qquad= \lim_{R\to\infty} \int_{\mb x\in S^{N-1}(b)}\int_{\mb x \cdot \mb p=0} f_q(\mb x,\mb p)\,\vp_{\mb x,R}(2p)\, \left(\Ga_N \theta^{N-1}_{b}\right)(2T,2p)\,dh^{N-1}_b(\mb x, \mb p)\\
		&\qquad= \int_{\mb x\in S^{N-1}(b)}\int_{\mb x \cdot \mb p=0} f_q(\mb x,\mb p)\, \left(\Ga_N \theta^{N-1}_{b}\right)(2T,2p)\,dh^{N-1}_b(\mb x, \mb p).\qedhere
	\end{align*}
\end{proof}
%%%%
We would like to imitate Thm.~\ref{thm:heat-powerseries}, i.e., to represent the heat (integral) operator as a power series of a exponential type. However, since the heat operator on the quadric $Q^{N-1}(b)\cong TS^{N-1}(b)$ is a double integral that requires first to integrate over the tangent plane at a point on the sphere and then to integrate over the sphere, the result needs some modification. 

\begin{prop}\label{prop:heat-integral}
	For any $T>0$ and any polynomial $q(\mb a,\ovl{\mb a})\in \cal P_k(\C)$ with $k<N$ we have 
	\begin{align*}
		\int_{Q^{N-1}(b)}q(\mb a(\mb x, \mb p),\ovl {\mb a}(\mb x, \mb p))\,d\nu^{N-1}_{b,T}(\mb x, \mb p)
		\hfill = \int_{S^{N-1}(b)} \left(e^{\frac T{b^2}\Ga_N}q\right)(\mb a(\mb x,\mb 0),\ovl {\mb a}(\mb x,\mb 0))\,d\ovl\sig^{N-1}_b(\mb x),
	\end{align*}
	where the exponential operator on the right-hand side is understood as a power series expansion of $\Ga_N$.
\end{prop}

We need to prove some auxiliary results first.
\begin{lem}\label{lem:J^2_a}
	Denote by $J^2_{\mb a,N}\big|_{Q^{N-1}(b)}$ the operator that first applies $J^2_{\mb a,N}$ then restricts to the quadric $Q^{N-1}(b)$. Then
	\begin{align}\label{eqn:Ja}
		J_{\mb a,N}^2\big|_{Q^{N-1}(b)}=-b^2\sum_{j=1}^N \frac{\pt^2}{\pt a_j^2}+(a\pt_a)^2_N
		+ (N-2)(a\pt_a)_N.
	\end{align}
\end{lem}

\begin{proof}
	Repeat the same computation in \cite[Sect.~3.6]{PeSen} with $x_j$'s replaced by $a_j$'s, $1\le j\le N$, we have 
	\begin{align*}
		J_{\mb a,N}^2&=-\sum_{1\le j<k\le N}\left(a_j\frac{\pt}{\pt a_k}-a_k\frac{\pt}{\pt a_j}\right)^2\\
		&=-\sum_{k=1}^N a_k^2\sum_{j=1}^N \frac{\pt^2}{\pt a_j^2}+\left(\sum_{j=1}^N a_j \frac{\pt}{\pt a_j}\right)^2
		+ (N-2)\sum_{j=1}^N a_j\frac{\pt}{\pt a_j}.
	\end{align*}
	When restricting to the quadric $Q^{N-1}(b)$, we have $\sum_{k=1}^N a_k^2=\mb a^2=b^2$, so we obtained the desired conclusion.
	%\begin{align*}
	%	J_{\mb a,N}^2\big|_{Q^{N-1}(b)}=-b^2\sum_{j=1}^N \frac{\pt^2}{\pt a_j^2}+\left(\sum_{j=1}^N a_j \frac{\pt}{\pt a_j}\right)^2
	%	+ (N-2)\sum_{j=1}^N a_j\frac{\pt}{\pt a_j}.\hspace{2em} \qedhere
	%\end{align*}
\end{proof}

\begin{lem}\label{lem:power-series-hyperbolic}
	Denote by $\Ga_N\big|_{Q^{N-1}(b)}$ the operator that first applies $\Ga_{N}$ and then restricts to the quadric $Q^{N-1}(b)$. Then, for any $N>k$,
	\begin{enumerate}[(i)]
		\item For any $f\in\cal P_k(\C)$, there exists a unique polynomial $g\in\cal P_k(\C)$ such that
		\begin{align*}
			\Ga_N|_{Q^{N-1}(b)}\, f = g|_{Q^{N-1}(b)}.
		\end{align*}
		Thus, we can consider $\Ga_{N}|_{Q^{N-1}(b)}$ as a map from $\cal P_k(\C)$ to itself.
		\item For any $q\in\cal P_k(\C)$ considered as a function on the quadric with $k<N$, the series 
		\begin{align*}
			e^{\frac{T}{b^2}\Ga_N} q:=\sum_{n=0}^\infty \frac{1}{n!}\left(\left.\frac{T}{b^2}\Ga_N\right|_{Q^{N-1}(b)}\right)^nq.
		\end{align*}
		is a polynomial in $\cal P_k(\C)$, so for any fixed $k<N$, the operator $e^{\frac{T}{b^2}\Ga_N}$ maps $\cal P_k(\C)$ to itself.
	\end{enumerate}
\end{lem}

Note that part (i) of the lemma is the complex counterpart (although we still keep the factor $b^2$) of Lem.~\ref{lem:maps-of-sph-lap}.

\begin{proof}
	\begin{enumerate}[(i)]
		\item We first prove that there exists a unique $h\in\cal P_k(\C)$ such that 
		$$J^2_{\mb a,N}|_{Q^{N-1}(b)}\, f = h|_{Q^{N-1}(b)}.$$ 
		For any $f\in\cal P_k(\R)$ (note that $f$ is not necessarily holomorphic), consider the polynomial $h$ given by the operator on the right-hand side of \eqref{eqn:Ja} applying on $f$. Note that $h$ only depends on $2k$ variables $a_1,\dots,a_k,\bar a_1,\cdots,\bar a_k$ as the differential operator $a_j\pt/\pt a_j$ for $j>k$ has no effect on $f$.
		
		Suppose there exists another polynomial $q\in \cal P_k(\C)$ such that $q|_{Q^{N-1}(b)}=h|_{Q^{N-1}(b)}$. Then $(q-h)|_{Q^{N-1}(b)}$ is identically $0$ on $Q^{N-1}(b)$. As $\mb a$ varies on $Q^{N-1}(b)$, the complex vector $(a_1,\cdots a_k)$ receives all possible values in $\C^k$, so if $a_j=u_j+iv_j$ $(u_j,v_j\in\R,1\le j\le k)$, then $q-h$ as a polynomial of $2k$ real variables $u_j,v_j$'s is $0$ in $\R^{2k}$, which implies $q\equiv h$.
		
		Defining the notation $J_{\ovl{\mb 	a},N}^2\big|_{Q^{N-1}(b)}$ in the same manner as $J_{\mb a,N}^2\big|_{Q^{N-1}(b)}$ in Lem.~\ref{lem:J^2_a}, we see that 
		\begin{align}%\label{eqn:Jabar}
			J_{\ovl{\mb a},N}^2\big|_{Q^{N-1}(b)}=
			-a^2\sum_{j=1}^N \frac{\pt^2}{\pt \bar a_j^2}+\left(\sum_{j=1}^N \bar a_j \frac{\pt}{\pt \bar a_j}\right)^2
			+ (N-2)\sum_{j=1}^N \bar a_j\frac{\pt}{\pt \bar a_j}.
		\end{align}
		and following the similar argument, $J^2_{\ovl{\mb a},N}$ maps $\cal P_k(\C)$ to itself.
		
		Finally, $\Ga_N|_{Q^{N-1}(b)}=\frac{1}{2}(J^2_{\mb a, N}+J^2_{\ovl{\mb a},N})|_{Q^{N-1}(b)}$ as a sum of two operators on $\cal P_k(\C)$. In particular, $g=\frac{1}{2}(J^2_{\mb a,N}|_{Q^{N-1}(b)}f+J^2_{\ovl{\mb a},N}|_{Q^{N-1}(b)}f)$ satisfies the statement.
		
		\item For any $q\in\cal P_k(\C)$, we can find a least $l$ such that $q\in\cal P_k^{\le l}(\C)$. As $\Ga_N|_{Q^{N-1}(b)}$ is a linear operator on a finite dimensional vector space $\cal P_{k}^{\le l}(\C)$, we can find a matrix representation for  $\Ga_N|_{Q^{N-1}(b)}$ and write the exponential operator as a power series expansion,
		\begin{align*}
			e^{\frac{T}{b^2}\Ga_N} q=\sum_{n=0}^\infty \left.\frac{1}{n!}\left(\frac{T}{b^2}\Ga_N\right|_{Q^{N-1}(b)}\right)^nq.
		\end{align*}
		as we desired.\qedhere
	\end{enumerate}
\end{proof}

%%%%%%%%
\begin{lem}[Exchanging integral with derivative {\cite[Thm.~2.27]{Fol95}}]\label{lem:interchange-lim-int}
	Let $(X,\mu)$ be a measure space, and $f(t,x)\colon \R^+\times X\to \R$ is integrable in $x$ and $C^1$ in $t$. Fix $t_0>0$. Suppose there exists $\ve\in(0,t_0)$ and a non-negative real-valued function $g_\ve\in L^1(X,\mu)$ such that $\big|\frac{\pt}{\pt t}f(t,x)\big |<g_\ve(x)$ for all $(t,x)\in (t_0-\ve,t_0+\ve)\times X$. Then
	\begin{align*}
		\left.\frac{\pt}{\pt t}\right|_{t=t_0}\int_{X} f(t,x)\,d\mu(x) = \int_{X} \left.\frac{\pt}{\pt t}\right|_{t=t_0}f(t,x)\,d\mu(x).
	\end{align*}
\end{lem}
\begin{comment}
\begin{proof}	
	We have 
	\begin{align*}
		\left.\frac{\pt}{\pt t}\right|_{t=t_0}\int_{X} f(t,x)\,d\mu(x) = \lim_{h\to 0} \frac{1}{h}\int_X [f(t+h,x)-f(t,x)]\,d\mu(x).
	\end{align*}
	For any $x\in X$, if $h$ satisfies $|h|<\ve$, there exists $c_h$ between $0$ and $h$ such that 
	\begin{align*}
		\frac{f(t_0+h,x)-f(t_0,x)}{h}=\left.\frac{d}{dt}\right|_{t_0+c_h} f(t,x).
	\end{align*} 
	The existence of $g_\ve$ guarantees that pointwisely,
	\begin{align*}
		\left|\frac{f(t_0+h,x)-f(t_0,x)}{h}\right|\le g(x).
	\end{align*}
	Since $g_\ve\in L^1(X,\mu)$, Dominated Convergence Theorem implies that we can exchange the limit with the integral, so that
	\begin{align*}
		\left.\frac{\pt}{\pt t}\right|_{t=t_0}\int_{X} f(t,x)\,d\mu(x) =  \int_X \lim_{h\to 0}\frac{1}{h} [f(t+h,x)-f(t,x)]\,d\mu(x) 
		= \left.\int_{X} \frac{\pt}{\pt t}\right|_{t=t_0} f(t,x)\,d\mu(x).
	\end{align*}
\end{proof}
\end{comment}
%%%%%%%%%%%%%%%%%%%%%%%%%%%%%
\begin{proof}[Proof of Prop.~\ref{prop:heat-integral}]
	Fix $k<N$ and for any $l\in \N$, define the functional $A_{b,T}^{k,l}\colon \cal P_{k}^{\le l}(\C) \to \C$ by the rule
	\begin{align*}
		A_{b,T}^{k,l}\,q=\int_{Q^{N-1}(b)} q\,d\nu^{N-1}_{b,T}.
	\end{align*} 
	From Lem.~\ref{prop:Gamma-p-function}, keeping in mind the correct time $2T$, we have 
	\begin{align*}
		\frac{\pt}{\pt T}\theta^{N-1}_{b}(2T,2p)=\left[\frac{\pt^2}{\pt p^2}+\frac{N-2}{b}\coth \left(\frac{p}{b}\right)\frac{\pt}{\pt p}\right]\theta^{N-1}_{b}(2T,2p)
		=\frac{1}{b^2}\left(\Ga_N \theta^{N-1}_{b}\right)(2T,2p).
	\end{align*}

	First, we note that the conditions of Lem.~\ref{lem:interchange-lim-int} are satisfied with $f(t,x)=\theta^{N-1}_b(t,x)$: following the same argument as the proof of Prop.~\ref{prop:int-by-part-hyperbolic}, as long as $T$ stays away from $0$, both terms
	\begin{align*}
		\frac{N-2}{b}\coth\left(\frac{p}{b}\right)\frac{\pt}{\pt p} \theta^{N-1}_b(2T,2p)\quad \tnm{and} \quad \frac{\pt^2}{\pt p^2} \theta^{N-1}_b(2T,2p)
 	\end{align*}
	are smooth near $p=0$ and dominated by $c_Ne^{-\frac{p^2}{T}}$ for some constant $c_N$ depending only on $N$ when $p$ is large (for the second term, we take another $\pt/\pt r$ of the recurrence relation in Lem.~\ref{lem:recur-relation} and use the estimate in Prop.~\ref{prop:estimate-heat-hyperbolic} again). 
	
	Applying Lem.~\ref{lem:interchange-lim-int}, that is, differentiating under the integral sign, we get
	\begin{align*}
		\frac{\pt}{\pt T}\int_{Q^{N-1}(b)}q\,d\nu^{N-1}_{b,T}
		&= \int_{Q^{N-1}(b)}q\,\left[\frac{\pt}{\pt T}\theta^{N-1}_{b}(2T,2p)\right]
		dh^{N-1}_{\sqrt N}\\
		& 
		= \frac {1}{b^2}\int_{Q^{N-1}}q\,
		\left(\Ga_N\theta^{N-1}_{b}\right)(2T,2p)
		\,dh^{N-1}_{\sqrt N}.
	\end{align*}
	Then applying Prop.~\ref{prop:int-by-part-hyperbolic} on the last part of the equality above, we have 
	\begin{align*}
		\frac{\pt}{\pt T}\int_{Q^{N-1}(b)}q\,d\nu^{N-1}_{b,T}
		= \frac{1}{b^2}\int_{Q^{N-1}(b)} (\Ga_Nq)\,d\nu^{N-1}_{b,T}.
	\end{align*}
	This leads to the differential equation
	\begin{align}\label{eqn:hyp-diff-eqn-1}
		\frac{\pt}{\pt T} A^{k,l}_{b,T}=\frac{1}{b^2}   A^{k,l}_{b,T}\Ga_N\big|_{Q^{N-1}(b)}
	\end{align}
	where we consider $\Ga_N\big|_{Q^{N-1}(b)}$ as a self-map on $\cal P^{\le l}_k(\C)$ from Lem.~\ref{lem:power-series-hyperbolic}(i). The initial condition is
	\begin{align}\label{eqn:hyp-diff-eqn-2}
		\lim_{T\downarrow 0} A^{k,l}_{b,T}=B^{k,l}_{b} 
	\end{align}
	where $B^{k,l}_{b}\colon \cal P^{\le l}_k(\C) \to \C$ is defined by
	\begin{align*}
		B^{k,l}_{b} q =\int_{S^{N-1}(b)} q(\mb a(\mb x, \mb 0),\ovl{\mb a}(\mb x,\mb 0))\,d\ovl{\sig}^{N-1}_b(\mb x).
	\end{align*} 
	
	Using knowledge from elementary Calculus, we can speculate that the solution is
	\begin{align*}
		A^{k,l}_{b,T} = B^{k,l}_{b}e^{\frac{T}{b^2}\Ga_N}
		=B^{k,l}_{b}\sum_{n=0}^\infty \left.\frac{1}{n!}\left(\frac{T}{b^2}\Ga_N\right)^n\right|_{\cal P^{\le l}_k(\C)},
	\end{align*}
	with the reminder that $\Ga_N$ when acting on $\cal P^{\le l}_k(\C)$ depends only on $k$ and not $N$. It is easy to verify that the speculated $A^{k,l}_{b,T}$ satisfy the equations \eqref{eqn:hyp-diff-eqn-1} and \eqref{eqn:hyp-diff-eqn-2}. Now, this proof obviously applies for any $l\in\N$, so for any $q\in\cal P_k(\C)$,
	\begin{align*}
		\int_{Q^{N-1}(b)}q(\mb a(\mb x, \mb p),\ovl {\mb a}(\mb x, \mb p))\,d\nu^{N-1}_{b,T}(\mb x, \mb p)
		= \int_{S^{N-1}(b)} \left(e^{\frac {T}{b^2}\Ga_N}q\right)(\mb a(\mb x,\mb 0),\ovl {\mb a}(\mb x,\mb 0))\,d\ovl\sig^{N-1}_b(\mb x).
	\end{align*}
	as we want.
\end{proof}

\smallskip

\begin{proof}[Proof of Thm.~\ref{thm:limitquad}]
	%Let $q\in \cal P_k(\C)$.
	Using Lem.~\ref{lem:J^2_a} for $J^2_{\mb a,N}$ and the conjugate counterpart $J^2_{\ovl{\mb a},N}$ we have
	\begin{align*}
		\frac{1}{N}\Ga_N\big|_{Q^{N-1}(\sqrt N)} &= -\frac{1}{2}\sum_{j=1}^N\left(\frac{\pt^2}{\pt a_j^2}+\frac{\pt^2}{\pt \bar a_j^2}\right)+
		\frac{1}{2}\sum_{j=1}^N\left(a_j\frac{\pt}{\pt a_j}+\bar a_j\frac{\pt}{\pt \bar a_j}\right)\\
		&\quad +\frac{1}{2N}\left(\left(\sum_{j=1}^N a_j \frac{\pt}{\pt a_j}\right)^2+\left(\sum_{j=1}^N \bar a_j \frac{\pt}{\pt \bar a_j}\right)^2 - 2\sum_{j=1}^N a_j\frac{\pt}{\pt a_j} - 2\sum_{j=1}^N \bar a_j\frac{\pt}{\pt \bar a_j}\right).
	\end{align*}
	
	We introduce the differential operator 
	\begin{align*}
		G_\infty :=-\frac{1}{2}\sum_{j=1}^\infty\left(\frac{\pt^2}{\pt a_j^2}+\frac{\pt^2}{\pt \bar a_j^2}\right)+
		\frac{1}{2}\sum_{j=1}^\infty\left(a_j\frac{\pt}{\pt a_j}+\bar a_j\frac{\pt}{\pt \bar a_j}\right),
	\end{align*} 
	For any polynomial $q$ in a finite dimensional vector space $\cal P^{\le l}_k(\C)$ , we have 
	Therefore,
	\begin{align*}
		\lim_{N\to \infty}\frac{1}{N}\Ga_N|_{Q^{N-1}(\sqrt N)} q = G_\infty q ,\qquad q\in \cal P^{\le l}_k(\C).
	\end{align*}
	 The convergence is  with respect to an arbitrary norm on $\cal P^{\le l}_k(\C)$. For any $k\in\N^*$, we also define the operator
	\begin{align*}
		G_k:=-\frac{1}{2}\sum_{j=1}^k\left(\frac{\pt^2}{\pt a_j^2}+\frac{\pt^2}{\pt \bar a_j^2}\right)+
		\frac{1}{2}\sum_{j=1}^k\left(a_j\frac{\pt}{\pt a_j}+\bar a_j\frac{\pt}{\pt \bar a_j}\right).
	\end{align*}
	Just as $(a\pt_a)_\infty$ is compatible with $(a\pt_a)_k$, $G_\infty$ is compatible with $G_k$ in the sense $G_\infty q= G_k q$ for all $q\in\cal P_k(\C)$.

	With $a_j=u_j+iv_j$ and $\bar a_j=u_j+iv_j$, and a little bit of algebra, we have
	\begin{align*}
		G_k%&=-\frac18\sum_{j=1}^k \left(\left(\frac{\pt}{\pt u_j}-i\frac{\pt}{\pt v_j}\right)^2+\left(\frac{\pt}{\pt u_j}+i\frac{\pt}{\pt v_j}\right)^2\right)+\\
		%&\hspace{3em}+\frac{1}{4}\sum_{j=1}^k\left((u_j+iv_j)\left(\frac{\pt}{\pt u_j}-i\frac{\pt}{\pt v_j}\right)+(u_j-iv_j)\left(\frac{\pt}{\pt u_j}+i\frac{\pt}{\pt v_j}\right)\right)\\
		%&=\frac{1}{4}\sum_{j=1}^k\left(\frac{\pt^2}{\pt v_j^2}-\frac{\pt^2}{\pt u_j^2}\right)+
		%\frac{1}{2}\sum_{j=1}^k\left(u_j\frac{\pt}{\pt u_j}+v_j\frac{\pt}{\pt v_j}\right)\\
		&=\frac{1}{4}\sum_{j=1}^k\left(-\frac{\pt^2}{\pt u_j^2}+2u_j\frac{\pt}{\pt u_j}\right)+
		\frac{1}{4}\sum_{j=1}^k\left(\frac{\pt^2}{\pt v_j^2}+2v_j\frac{\pt}{\pt v_j}\right).
	\end{align*}
	For now we write for simplicity $u=|\mb u|$, $\Del_{\R^N,u}=\sum_{j=1}^N \frac{\pt^2}{\pt u_j^2}$, $u\pt_u=\sum_{j=1}^N u_j\frac{\pt}{\pt u_j}$ and similarly for $v,\Del_{\R^N,v}$, and $v\pt_v$. Note that any of the listed operators of $u$ will commute with the mentioned operators of $v$.
	
	\medskip 
	
	Fix a $q(\mb a,\ovl{\mb  a})\in \cal P^{\le l}_k(\C)$, note that $e^{\frac {T}{N}\Ga_N}q$ is still a polynomial of variables $a_j,\bar a_j,\ j=1,\dots, k$ by Lem.~\ref{lem:power-series-hyperbolic}(ii). Thus, the continuity of the exponential operator implies 
	\begin{align}\label{eqn:lim-exp-Ga-N}
		\lim_{N\to \infty}e^{\frac{T}{N}\Ga_N} q = e^{TG_k} q.
	\end{align}
	
	By Prop.~\ref{prop:heat-integral}, the convergence of the spherical measure $\ovl{\sig}^{N-1}$ to the Gaussian $\mu^\infty_1$ in Thm.~\ref{thm:limitsph}, and \eqref{eqn:lim-exp-Ga-N}, we then have
	\begin{align*}
		&\lim_{N\to\infty} \int_{Q^{N-1}(\sqrt{N})} q(\mb a,\ovl{\mb a})\,d\nu^{N-1}_{T}
		=\lim_{N\to\infty} \int_{S^{N-1}(\sqrt{N})} \left(e^{\frac{T}{N}\Ga_N}q\right)(\mb a(\mb x,\mb 0),\ovl{\mb a}(\mb x,\mb 0))\,d\ovl{\sig}^{N-1}(\mb x)\\
		&=\int_{\R^\infty} \left(e^{TG_\infty}q\right)(\mb a(\mb x,\mb 0),\ovl{\mb a}(\mb x,\mb 0))\,d\mu^\infty_1(\mb x)
		=(2\pi)^{-\frac{k}{2}}\int_{\R^k} \left(e^{TG_k}q\right)(\mb a(\mb x,\mb 0),\ovl{\mb a}(\mb x,\mb 0))\,e^{-\frac{\mb x^2}2}\,d\mb x\\
		&=(2\pi)^{-\frac{k}{2}}\int_{\R^k} \left(e^{TG_k}q\right)(\mb u+i\mb 0,\mb u-i\mb 0)\,e^{-\frac{\mb u^2}2}\,d\mb u.
	\end{align*}
	The last equality above comes from the fact that $\mb u=\cosh(\frac {p}{b})\mb x$ and $\mb v= \frac{b}{p}\sinh(\frac{p}{b})\mb p$, so if $p=0$, then $\mb u=\mb x$ and $\mb v=\mb 0$.	By Thm.~\ref{thm:heat-real-general}, to integrate a function with respect to the Gaussian $d\mu^N_1$ is to apply the heat operator $e^{\frac{1}{2}\Del_{\R^k}}$ to that function and evaluate it at $\mb 0$, so we have 
	\begin{align*}
		&(2\pi)^{-\frac{k}{2}}\int_{\R^k} \left(e^{TG_k}q\right)(\mb u+i\mb 0,\mb u-i\mb 0)\,e^{-\frac{\mb u^2}2}\,d\mb u
		= \left.e^{\frac{1}{2}\Del_{\R^k,u}}\left(e^{TG_k} q(\mb u+i\mb 0,\mb u-i\mb 0)\right)\right|_{\mb u=\mb 0}\\
		&= \left.\left(e^{\frac{1}{2}\Del_{\R^k,u}}e^{TG_k} q\right)(\mb u+i\mb v,\mb u-i\mb v)\right|_{\mb u=\mb v=\mb 0}.
	\end{align*}
	
	Combining what we have so far, we get
	\begin{align}
		\lim_{N\to\infty} \int_{Q^{N-1}(\sqrt N)} q\,d\nu^{N-1}_T = \left.\left(e^{\frac{1}{2}\Del_{\R^k,u}}e^{TG_k} q\right)(\mb u+i\mb v,\mb u-i\mb v)\right|_{\mb u=\mb v=\mb 0}.\label{eqn:limit-measure-first-answer}
	\end{align}

	This basically is the answer to what the large-$N$ limit of the quadric measure is in terms of exponential operators. For the rest of the proof, we will express the product of exponential operators with the evaluation $\mb u=\mb v=\mb0$ on the right-hand side as a Gaussian measure using the Baker--Campbell--Hausdorff formula in Lem.~\ref{lem:exercise}.
	
	Now, the operators $e^{\frac{1}{2}\Del_{\R^k,u}}$ and $e^{TG_k}$ act boundedly on the vector space $\cal P^{\le l}_k(\C)$, so we can consider them as matrices and compute $e^{\frac{1}{2}\Del_{\R^k,u}}e^{TG_k}$. We have 
	\begin{align*}
		G_k=\frac{1}{4}\left(-\Del_{\R^k,u}+2(u\pt_u)_k+\Del_{\R^k,v}+2(v\pt_v)_k\right),
	\end{align*}
	and $[G_k,\Del_{\R^k,u}]=\frac{1}{2}[(u\pt_u)_k,\Del_{\R^k,u}] = -\Del_{\R^k,u}$ by Lem.~\ref{lem:commutative-relation}, so we use the second identity of Lem.~\ref{lem:exercise} for $X=TG_k$, $Y=\frac{1}{2}\Del_{\R^k,u}$, $\al=-T$ to obtain
	\begin{align*}
		e^{\frac{1}{2}\Del_{\R^k,u}}e^{TG_k} &= e^{TG_k+\frac{T}{2(1-e^{-T})}\Del_{\R^k,u}}\\
		%& = \exp\left[\frac{T}{4}\left(\left(-1+\frac{2}{1-e^{-T}}\right)\Del_{\R^k,u}+2(u\pt_u)_k+\Del_{\R^k,v}+2(v\pt_v)_k\right)\right]\\
		& = \exp\left[\frac{T}{4}\left(\frac{1+e^{-T}}{1-e^{-T}}\Del_{\R^k,u}+2(u\pt_u)_k\right)\right]\exp\left[\frac{T}{4}\left(\Del_{\R^k,v}+2(v\pt_v)_k\right)\right].
	\end{align*}
	Now, using the third identity of Lem.~\ref{lem:exercise} with $X=\frac{T}{2}(v\pt_v)_k$, $Y=\frac{T}{4}\frac{1+e^{-T}}{1-e^{-T}}\Del_{\R^k,u}$, and $\al=-T$, we have  
	\begin{align*}
		\exp\left[\frac{T}{4}\left(\frac{1+e^{-T}}{1-e^{-T}}\Del_{\R^k,u}+2(u\pt_u)_k\right)\right]
		&= e^{\frac{T}{2}(u\pt_u)_k} \exp\left[\frac{1-e^{T}}{-4}\cdot\frac{1+e^{-T}}{1-e^{-T}}\Del_{\R^k,u}\right]\\
		&= e^{\frac{T}{2}(u\pt_u)_k} e^{\frac{e^{T}+1}{4}\Del_{\R^k,u}},
	\end{align*}
	Use the third identity another time with $X=\frac{T}{2}(u\pt_u)_k$, $Y=\frac{T}{4}\Del_{\R^k,v}$, and $\al=-T$, we deduce  
	\begin{align*}
		e^{\frac{T}{4}\left(\Del_{\R^k,v}+2(v\pt_v)_k\right)} 
		= e^{\frac{T}{2}(v\pt_v)_k} e^{\frac{e^T-1}{4}\Del_{\R^k,v}}.
	\end{align*}
	Consequently,
	\begin{align*}
		\left.\left(e^{\frac{1}{2}\Del_{\R^k,u}}e^{TG_k} q\right)\right|_{\mb u=\mb v=\mb 0}
		=  \left.\left( e^{\frac{T}{2}(u\pt_u)_k} e^{\frac{1+e^{T}}{4}\Del_{\R^k,u}}e^{\frac{T}{2}(v\pt_v)_k} e^{\frac{e^T-1}{4}\Del_{\R^k,v}} q\right)\right|_{\mb u=\mb v=\mb 0}
	\end{align*}

	Now, we recall from \ref{def:dilation} that $e^{\frac{T}{2}(u\pt_u)_k}$ and $e^{\frac{T}{2}(v\pt_v)_k}$ act on $q$ by dilating $\mb u\mapsto e^{T/2}\mb u$ and $\mb v\mapsto e^{T/2}\mb v$, respectively, so by letting $n_T=\pi^2(e^{2T}-1)$ together with Thm.~\ref{thm:heat-real-general}, we have 
		\begin{align*}
			&\left.e^{\frac{T}{2}(u\pt_u)_k} e^{\frac{e^T+1}{4}\Del_{\R^k,u}}e^{\frac{T}{2}(v\pt_v)_k} e^{\frac{e^T-1}{4}\Del_{\R^k,v}} q\right|_{\mb u=\mb v=\mb 0}\\
			&= \left.n_T^{-\frac{k}{2}}e^{\frac{T}{2}(u\pt_u)_k}e^{\frac{T}{2}(v\pt_v)_k}\int_{\R^{2k}} q(\mb u'+i\mb v',\mb u'-i\mb v')e^{-\frac{(\mb u-\mb u')^2}{e^T+1}-\frac{(\mb v-\mb v')^2}{e^T-1}}\,d\mb u'\,d\mb v'\right|_{\mb u=\mb v=\mb 0}\\
			&= n_T^{-\frac{k}{2}} \left.\int_{\R^{2k}} q(\mb u'+i\mb v',\mb u'-i\mb v')e^{-\frac{(e^{T/2}\mb u-\mb u')^2}{e^T+1}-\frac{(e^{T/2}\mb v-\mb v')^2}{e^T-1}}\,d\mb u'\,d\mb v'\right|_{\mb u=\mb v=\mb 0}\\
			&= n_T^{-\frac{k}{2}} q(\mb u'+i\mb v',\mb u'-i\mb v')e^{-\frac{{\mb u'}^2}{e^T+1}-\frac{{\mb v'}^2}{e^T-1}}\,d\mb u'\,d\mb v'.
		\end{align*}
	Therefore, \eqref{eqn:limit-measure-first-answer} implies
	\begin{align*}
		\lim_{N\to\infty} \int_{Q^{N-1}(\sqrt{N})} q(\mb a(\mb x,\mb p),\ovl{\mb a}(\mb x,\mb p))\,d\nu^{N-1}_{T}
		&=\int_{\C^\infty} q(\mb u+i\mb v,\mb u-i\mb v)\,d\ga^\infty_{T}(\mb u,\mb v)\\
		&=\int_{\C^k} q(\mb u+i\mb v,\mb u-i\mb v)\,d\ga^k_{T}(\mb u,\mb v)
	\end{align*}
	from which we can easily derive the equality
	\begin{align*}
		\lim_{N\to\infty} \int_{Q^{N-1}(\sqrt{N})} q_1(\mb a(\mb x,\mb p),\ovl{\mb a}(\mb x,\mb p))\, \ovl{q_2(\mb a(\mb x,\mb p),\ovl{\mb a}(\mb x,\mb p))}\,d\nu^{N-1}_{T}=
		\la q_1,q_2\ra_{L^2(\C^\infty,\ga_T^\infty)}
	\end{align*}
	for all $q_1,q_2\in\cal P_k(\C)$. %by noting that a product of polynomials in $(\mb a,\ovl{\mb a})$ is again a polynomial of the same variables.
\end{proof}
%----------------------------------------------------------------------------------
\subsection{Large-$N$ limit of the Segal--Bargmann transforms}\label{sec:limitmap}
%----------------------------------------------------------------------------------
Recall the Segal--Bargmann transform on $L^2(S^{N-1}(\sqrt N),\ovl\sig^{N-1})$ from Def.~\ref{def:SB-sphere}. Let us now combine the results in Sects.~\ref{sec:limit-horizontal} and \ref{sec:limit-quad-measure} to establish the large-$N$ limit of the Segal--Bargmann transforms on the spheres, thus showing the commutativity of the diagram at the end of Sect.~\ref{sec:main-res}.
%%%%%%%%%%%%%
\begin{thm}\label{thm:limitmap}
	Let $q\in \cal P_k(\R)$ be a real polynomial of $k$ variables $x_1\dots, x_k$ considered as a function of $N>k$ variables. Then
	\begin{align*}
		\lim_{N\to\infty} \cal C^{N-1}_{T}\,q=\left(e^{-\frac{T}{2}r\pt_r}e^{\frac{1-e^{-T}}{2}\Del_{\R^k}}q\right)_\C=\mb B_T\,q.
	\end{align*}
\end{thm}
\begin{proof}
	First, we need to understand what the equality in the statement of the theorem means. We recall that from Lem.~\ref{lem:maps-of-sph-lap} that the operator $\Del_{S^{N-1}(\sqrt N)}$ can be considered as a map from  $\cal P_k(\R)$ to itself, and since for any $q\in\cal P_k(\R)$
	\begin{align*}
		\cal C^{N-1}_{T}\, q= \left(e^{\frac{T}{2}\Del_{S^{N-1}(\sqrt N)}}\, q\right)_\C .
	\end{align*}
	the map $\cal C^{N-1}_{T}$ can also be viewed as self-map on $\cal P_k(\R)$. The right-hand side can also be expressed as a power series, from Thm.~\ref{prop:heat-powerseries-sphere}. Hence, by the continuity of the exponential operator and Prop.~\ref{prop:Sen-Hermite}, we have 
	\begin{align*}
		\lim_{N\to\infty} \cal C^{N-1}_{T}\,q = \left({e^{\frac{T}{2}\mb H}} q\right)_\C 
		=\left(e^{\frac{T}{2}(\Del_{\R^k}-(r\pt_r)_k)}\,q\right)_\C.
	\end{align*}
	The second equality above comes from the compatibility of $\mb H$ with $\Del_{\R^k}-(r\pt_r)_k$. 
	From the proof of Thm.~\ref{thm:horizontal_arrow}, we have 
	\begin{align*}
		\mb B_T\, q= \left(e^{\frac{T}{2}\mb H}q\right)= \left(e^{-\frac{T}{2}(r\pt r)_k} e^{\frac{1-e^{-T}}{2}\Del_{\R^k}}\,q\right)_\C,
	\end{align*}
	so
	\begin{align*}
		\lim_{N\to\infty} \cal C^{N-1}_{T}\,q=\mb B_T\, q
	\end{align*}
	as we wanted.
\end{proof}

\bibliographystyle{plain}
\bibliography{thesis-cite.bib}    
\end{document}